\documentclass[11pt, a4paper]{amsart}
\usepackage[margin = 2.7cm]{geometry}
\usepackage[title]{appendix}
\usepackage{mathtools}
\usepackage{setspace}
\usepackage{textcomp}
\usepackage{enumerate}
\usepackage{graphicx}
\usepackage{subfigure}
\usepackage[dvipsnames]{xcolor}

\usepackage{graphics}
\usepackage[colorlinks=true]{hyperref}
\hypersetup{urlcolor=blue, citecolor=red}

\usepackage{amsmath}
\usepackage{mathrsfs}

\newtheorem{thm}{Theorem}[section]
\newtheorem{crl}{Corollary}

\newtheorem{lm}{Lemma}
\newtheorem{pp}{Proposition}

\theoremstyle{definition}
\newtheorem{df}{Definition}
\newtheorem{remark}{Remark}

\newcommand*\diff{\mathop{}\!\mathrm{d}}

\newcommand{\ep}{\epsilon}
\newcommand{\lge}{\langle}
\newcommand{\rge}{\rangle}

\newcommand{\M}{(0, T) \times \Omega}

\newcommand{\op}{(\partial^2_t - \Delta - \partial_t\Delta )}
\newcommand{\opq}{(\partial^2_t - \Delta - \partial_t\Delta  + q(t,x))}
\newcommand{\opqstar}{(\partial^2_t - \Delta + \partial_t\Delta  + q(t,x))}

\newcommand{\Rn}{\mathbb{R}^n}

\newcommand{\Tzero}{T_\omega}

\newcommand{\Tone}{T_1}
\newcommand{\Tonet}{\widetilde{T}_1}
\newcommand{\Ttwo}{T_2}
\newcommand{\Ttwot}{\widetilde{T}_2}

\DeclareMathOperator\supp{supp}

\newcommand{\dt}{\partial_t}
\newcommand{\ddt}{\partial^2_t}

\newcommand{\vN}{v_{N}}

\newcommand{\betak}{\beta^{(k)}}
\newcommand{\betaone}{\beta^{(1)}}
\newcommand{\betatwo}{\beta^{(2)}}
\newcommand{\qone}{{q^{(1)}}}
\newcommand{\qtwo}{{q^{(2)}}}

\newcommand{\tP}{\widetilde{P}}

\newcommand{\zetaz}{\zeta^o}

\newcommand{\ba}{\[\begin{aligned}}
\newcommand{\ea}{\end{aligned}\]}

\newcommand{\Bpk}{B_{p, k}}
\newcommand{\Bpkloc}{B_{p, k}^\mathrm{loc}}
\newcommand{\BinP}{B_{\infty, \tP}}
\newcommand{\BinPloc}{B_{\infty, \tP}^\mathrm{loc}}
\newcommand{\Po}{P_o}
\newcommand{\mJ}{\mathcal{J}}
\newcommand{\mI}{\mathcal{I}}
\newcommand{\vep}{\varepsilon}
\newcommand{\mL}{{L}}

\newcommand{\qbone}{{\qone, \betaone}}
\newcommand{\qbtwo}{{\qtwo, \betatwo}}
\newcommand{\qbk}{{q^{(k)}, \beta^{(k)}}}
\newcommand{\mLlin}{\mL^\mathrm{lin}}

\newcommand{\uk}{u^{(k)}}
\newcommand{\uone}{u^{(1)}}
\newcommand{\utwo}{u^{(2)}}

\newcommand{\Lamlin}{\Lambda^\mathrm{lin}}

\newcommand{\teal}{\color{teal}}
\newcommand{\Pad}{^t \! P}
\newcommand{\tPad}{\widetilde{{^t \! P}}}
\newcommand{\refb}{\hyperref[condition_b]{b}}
\newcommand{\tX}{\widetilde{X}}
\newcommand{\tf}{\tilde{f}}
\newcommand{\tu}{{u}}
\usepackage{accents}
\newlength{\dhatheight}
\newcommand{\doublehat}[1]{%
    \settoheight{\dhatheight}{\ensuremath{\hat{#1}}}%
    \addtolength{\dhatheight}{-0.35ex}%
    \hat{\vphantom{\rule{1pt}{\dhatheight}}%
        \smash{\hat{#1}}}}
\newcommand{\epv}{r_0}
\newcommand{\Laplace}{\Delta}
\newcommand{\ltwo}{{L^2}}
\newcommand{\hone}{{H^1}}

\newcommand{\kn}{k_n}
\newcommand{\zm}{{Z^m}}
\newcommand{\zmm}{{Z^{m-1}}}

\newcommand{\Cjl}{C_{l,j}}
\newcommand{\Ckn}{\lceil (\kn-1)/2\rceil}
\newcommand{\Ngf}{N_{fg}}
\newcommand{\tzm}{{\widetilde{Z}^m}}

\newcommand{\un}{u_i}
\newcommand{\tg}{\tilde{g}}
\newcommand{\hEz}{E_1}
\newcommand{\hE}{E_H}
\newcommand{\Gm}{\mathcal{G}^m}
\newcommand{\Hm}{\mathcal{H}^m}
\newcommand{\qCm}{{{\teal Q}}}

\title[Inverse problems for a quasilinear strongly damped wave equation]{Inverse problems for a quasilinear strongly damped wave equation arising in nonlinear acoustics}
\author[Li Li]{Li Li }
\address[1]{Department of Mathematics, University of California, Irvine, CA 92697, USA}

\author[Yang Zhang]{Yang Zhang }
\address[2]{Department of Mathematics, University of Washington,
    Seattle, WA 98195, USA}

\date{}

\begin{document}

\maketitle

\begin{abstract}
We consider inverse problems for a Westervelt equation with a strong damping and a time-dependent potential $q$.
We first prove that all boundary measurements, including the initial data, final data, and the lateral boundary measurements, uniquely determine  $q$ and the nonlinear coefficient
$\beta$.
The proof is based on complex geometric optics construction and the approach  proposed by Isakov.
Further, by considering fundamental solutions supported in a half-space constructed by H\"ormander, we prove that with vanishing initial conditions the Dirichlet-to-Neumann map determines $q$ and $\beta$.
\end{abstract}

\section{Introduction}
The Westervelt equation is a classical model in the field of nonlinear acoustics.
Within a thermoviscous medium, in some situations the Westervelt equation is accompanied by a strong damping term, as exemplified in \cite{kaltenbacher2018fundamental}.
In this work, we consider inverse problems of recovering a potential and a nonlinear coefficient for the strongly damped Westervelt equation, from two types of boundary measurements.

More explicitly,
for $n \geq 2$, let $T > 0$ be fixed  and $\Omega \subset \mathbb{R}^n$ be a bounded set with a smooth boundary $\partial \Omega$.
Consider the following quasilinear strongly damped wave equation
\begin{equation}\label{eq_nl}
    \begin{aligned}
        \opq u - \beta(t,x) \partial_t^2 (u^2) &= 0, & \  & \mbox{on } \M,\\
        u(t,x) &= h, & \ &\mbox{for  } x\in \partial \Omega,\\
        u = g_0, \  \partial_t u &= g_1, & \ &\mbox{for  } t=0,\\
    \end{aligned}
\end{equation}
where $q$ is the potential, $\beta$ is the nonlinearity coefficient for the medium,
and $\Delta $ is the Laplacian operator with respect to $x$.
This initial-boundary value problem is locally well-posed for sufficiently small data $(g_0, g_1, h)$ satisfying the compatibility condition, see Proposition \ref{pp_nl}.



\subsection*{Main results} We introduce the following notations.
For fixed large $m$,
let $\mathcal{G}^m$ be the space of all
\[
(g_0, g_1, h) \in H^{2m+1}(\Omega) \times H^{2m-1}(\Omega) \times C^{2m + 2}([0,T] \times \partial \Omega)
\]satisfying the $m$th-order compatibility condition defined in (\ref{def_comp_nl}).
Then there exists a unique solution $u$ to (\ref{eq_nl}) for sufficiently small $(g_0, g_1, h) \in \Gm$.
First, we consider the recovery from the so-called all boundary measurements.
We define
\[
\mL_{q, \beta}: (g_0, g_1, h) \rightarrow (u(T), \partial_t u(T), \partial_\nu u|_{\partial \Omega}),
 \]
for small initial-boundary data $(g_0, g_1, h) \in \Gm$.
We prove that the following result.
\begin{thm}\label{thm1}
    Let $n \geq 2$, $m \geq 2(\lfloor n/2 \rfloor +1)$, and $T > 0$ be fixed.
    Let $\Omega \subset \Rn$ be a bounded open set with smooth boundary.
    For $k = 1,2$, suppose $q^{(k)}, \beta^{(k)} \in C^\infty([0,T] \times \bar{\Omega})$ with $\dt^j \beta(0,x) = 0$ for $x \in \partial \Omega$ and $j = 0, \ldots, m$.
    Suppose
    \[
    L_{\qone, \betaone}(g_0, g_1, h) = L_{\qtwo, \betatwo}(g_0, g_1, h)
    \]
    for small $(g_0, g_1, h) \in \Gm$,
    then
    \[
    \qone = \qtwo, \qquad \betaone = \betatwo, \qquad \text{for } (t,x)\in [0,T] \times \bar{\Omega}.
    \]
\end{thm}
To prove this result, we construct a exponentially growing geometric optics solution (i.e., the phase function has complex frequency) and
use  regular fundamental solutions constructed in \cite[Chapter 10]{H2005analysis} to estimate the remainder term.
This approach is first introduced in \cite{isakov1991completeness}, which provides an alternative proof of the unique determination theorem for the operator $\Laplace + q$ in \cite{sylvester1987global} and can be applied to other differential operators with constant coefficients.
This idea is also used in \cite{krupchyk2010borg} to prove a multidimensional Borg-Levinson theorem for elliptic operators of higher order with constant coefficients.
With the remainder term relatively small, for our model, the recovery of $q$ comes from the linear problem and  that of $\beta$ comes from the second-order linearization.

In many applications the initial data vanish, and in some cases measuring the final data is impossible.
To address this, in the second part of this paper, we consider the recovery of the potential $q$ and the nonlinear coefficient $\beta$ from lateral boundary measurements.
More precisely, we define
\[
\Hm = \{h \in C^{2m+2}([0,T]\times \partial \Omega): \ \dt^j h(0,x) = 0, \text{ for } x \in \partial \Omega \text{ and } j = 0, \ldots, m\}.
\]
When $g_0 = g_1 = 0$, the boundary value problem (\ref{eq_nl})
has a unique solution for small $h \in \Hm$, as a special example of Proposition \ref{pp_nl}.
We consider the Dirichlet-to-Neumann map
\[
\Lambda_{q, \beta}: h \rightarrow \partial_\nu u|_{\partial \Omega},
\]
where $u$ satisfies (\ref{eq_nl}) with $g_0 = g_1 =0$.
We prove the following result.
\begin{thm}\label{thm2}
    Let $n \geq 2$, $m \geq 2(\lfloor n/2 \rfloor +1)$, and $T > 0$ be fixed.
    Let $\Omega \subset \Rn$ be a bounded and convex open set with smooth boundary.
    For $k = 1,2$, suppose $q^{(k)}, \beta^{(k)} \in C^\infty([0,T] \times \bar{\Omega})$.
    Suppose
    \[
    \Lambda_{\qone, \betaone}(h) = \Lambda_{\qtwo, \betatwo}(h)
    \]
    for small $h \in \Hm$,
    then
    \[
    \qone = \qtwo, \qquad \betaone = \betatwo, \qquad \text{for } (t,x)\in [0,T] \times \bar{\Omega}.
    \]
\end{thm}
To prove this theorem, we construct a complex geometric optics solution and would like to follow the same idea that we use earlier.
To construct a remainder term with vanishing initial conditions, we consider the characteristic Cauchy problems studied in \cite[Chap 12]{H2005analysis}.
We observe that for the differential operator involved in the equation for the remainder term, see (\ref{eq_wR_new}),
the boundary of the half-space $H = \{t \geq \vep \}$ is characteristic, where $\vep >0$ is small.
Then there is no uniqueness for the initial value problem and the aim is to find a solution with support in $H$ (equivalently, a fundamental solution with support in $H$).
We describe under which condition  such fundamental solution almost exists and particularly with a desired estimate,
see Proposition \ref{pp_12813} and \ref{pp_Eestimate_new}.
Then we verify for our model, this condition is satisfied and therefore the remainder term is relatively small.
Then the recovery of $q$ and $\beta$ follows from a similar argument as before.

We note that in these two theorems above, the assumption that $q$ and $\beta$ are smooth is just for convenience.
One may consider non-smooth functions but with certain regularity, which mainly depends on the assumptions for the local well-posedness.



Inverse problems for Westervelt equations has been considered in \cite{acosta2022nonlinear, eptaminitakis2022weakly}, for a general nonlinearity in \cite{uhlmann2023inverse},
and with various damping effects in
\cite{kaltenbacher2021identification,kaltenbacher2023simultaneous,kaltenbacher2022inverse,kaltenbacher2023nonlinearity,zhang2023nonlinear,li2023inverse}.
Most recently in \cite{fu2023inverse}, the recovery of a time-dependent nonlinearity for the Jordan-Moore-Gibson-Thompson equation is considered, from all boundary measurements.
The model considered in \cite{fu2023inverse} has the principal part given by $\dt^3 - b \dt \Laplace$, for a constant $b>0$.
In this paper, the principal part of our model is given by the strong damping $\dt \Laplace$,
which is not hypoelliptic or of principal type.

The idea of using multi-fold linearization and nonlinear interaction of solutions to solve inverse problems can be found in \cite{kurylev2018inverse} for a semilinear wave equation, see also \cite{lassas2018inverse}.
In \cite{acosta2022nonlinear}, the authors use Gaussian beams and the second-order linearization to recover the nonlinear coefficient for a Westervelt equation without damping.
For inverse problems for time-dependent nonlinear Schr\"odinger equations,
in \cite{lassas2022inverse} the authors use geometric optics construction and the second-order linearization to recover a potential and the nonlinear coefficient; later in \cite{lai2023partial}, the stability and partial data problem are considered. The multi-fold linearization method has also been used to solve inverse problems for semilinear elliptic and fractional equations (see \cite{krupchyk2020remark, lassas2020partial, li2023elas} for instance).

The rest of this paper is organized as follows.
In Section \ref{sec_well},  we consider the initial-boundary value problem in (\ref{eq_nl}) and prove the local well-posedness.
In Section \ref{sec_prelim}, we briefly recall some  results in \cite[Chap 10, 12]{H2005analysis} and
we prove Proposition \ref{pp_E0}, \ref{pp_12813}, \ref{pp_Eestimate_new}, which provides estimates for regular fundamental solutions and fundamental solutions with support in a half space.
Note that these results may be used for other differential operators with constant coefficients.
In Section \ref{sec_goc}, we construct geometric optic solutions using a complex linear phase function.
In Section \ref{sec_allbd}, we prove Theorem \ref{thm1} using regular fundamental solutions.
In Section \ref{sec_gocnew}, we prove Theorem \ref{thm2} using fundamental solutions with support in a half space.
In both sections, we recover $q$ from the first-order linearization of the measurements and we recover $\beta$ from the second-order linearization.
In Appendix \ref{sec_apdx}, we prove the local well-posedness for the linearized problem, which is used in the proof of Section \ref{sec_well}.

\noindent \textbf{Acknowledgements.} L.L. and Y.Z. would like to thank Professor Katya Krupchyk and Professor Gunther Uhlmann for helpful discussions.

\section{The local well-posedness}\label{sec_well}
In this section, we prove the initial-boundary value problem (\ref{eq_nl}) is locally well-posed.
The local and global well-posedness for the strongly damped Westervelt equation without the potential $q(t,x)$
is proved in \cite{kaltenbacher2011well}, where properties of the semigroup generated by strongly damped wave equations are used.
Here we follow the ideas in \cite{dafermos1985energy} for our model.

For $r \in \mathbb{R}$,  we use $\lfloor r \rfloor$ to denote the largest integer that is no greater than $r$ and we use $\lceil r \rceil$ to denote the largest integer that is no less than $r$.
For convenience, we define
\begin{align}\label{eq_kn}
    \kn = \lfloor n/2 \rfloor +1.
\end{align}
The aim of this section is to establish local well-posedness for the following initial-boundary value problem
\begin{equation}\label{eq_nonlinear_bvp}
    \begin{aligned}
        \ddt u - \Laplace u - \dt \Laplace u  + q(t,x)u  - \beta(t,x) \partial_t^2( u^2) &= 0, & \  & \mbox{on } (0,T)\times \Omega,\\
        u &= h, & \ &\mbox{for  } x\in \partial \Omega,\\
        u = g_0(x), \quad  \partial_t u &= g_1(x) , & \ &\mbox{for  } t=0.
    \end{aligned}
\end{equation}
%
Note that the initial-boundary data $g_0,g_1, h$ must be compatible in some sense.
For this purpose, we recursively define
\begin{align}\label{def_gl_nl}
    g_{l+1} \coloneqq  (1-2\beta(0)g_0)^{-1} \tilde{G}_l(0,x,g_0, \ldots, g_l),
\end{align}
where we write
\begin{align}\label{def_tG}
    &\tilde{G}_l(t, x, u, \dt u, \ldots, \dt^l u)\\
    =& 
    -\sum_{j=2}^{l}\Cjl\dt^{l+1-j} (1-2\beta u) \dt^j u
    + \Laplace \dt^{l-1} u
    +  \Laplace \dt^{l} u
    - \dt^{l-1}(qu)
    -2 \dt^{l-1}(\beta (\dt u)^2). \nonumber
\end{align}
When $l=1$ we do not have the first term on the right-hand side.
We say $(g_0, g_1, h)$ satisfies the \textit{$m$th-order compatible condition}, if
\begin{align}\label{def_comp_nl}
    g_{l}  = \dt^{l} h(0), \qquad \text{for }l = 0, \ldots, m.
\end{align}
Recall $\mathcal{G}^m$ is the space of all $(g_0, g_1, h) \in H^{2m+1}(\Omega) \times H^{2m-1}(\Omega) \times C^{2m + 2}([0,T] \times \partial \Omega)$ satisfying the $m$th-order compatibility condition (\ref{def_comp_nl}).
We define
\[
\|(g_0, g_1, h)\|_{\Gm} = \|h\|_{C^{2m + 2}([0,T] \times {\partial \Omega})} + \|g_0\|_{H^{2m+1}}+ \|g_1\|_{H^{2m-1}}.
\]
Then we prove the following proposition.
\begin{pp}\label{pp_nl}
    For fixed $m \geq 2\kn$ and $T>0$,
    let $q, \beta \in C^{\infty}([0,T] \times \bar{\Omega})$. Suppose $(g_0, g_1, h) \in \Gm$ with
    $\|(g_0, g_1, h)\|_{\Gm} \leq \rho_h$.
    Then for sufficiently small $\rho_h >0$,
    the initial-boundary value problem (\ref{eq_nonlinear_bvp}) has a unique solution
    \[
    u \in \bigcap_{k=0}^{m} C^{m-k}([0,T]; H^{2k+1}(\Omega)),
    \]
    such that for some constant $C>0$, we have
    \begin{align}\label{est_nl}
     \sum_{k=0}^m \sup_{s \in [0, T]}   \|\partial_t^{m-k} u(s)\|^2_{{H^{2k+1}}}
    \leq C \|(g_0, g_1, h)\|_{\Gm}.
    \end{align}
\end{pp}
\begin{proof}
Note that for $h \in C^{2m + 2}((0,T) \times \partial \Omega)$, there exists a function $u_h \in C^{2m+1}([0, T] \times \bar{\Omega})$ such that
    \[
    u_h|_{\partial \Omega} = h
   \quad \text{with} \quad
    \| u_h (t)\|_{C^{2m+2}([0,T] \times \bar{\Omega})} \leq
    C \|h (t)\|_{C^{2m + 2}([0,T] \times \partial \Omega)}.
    \]
   In particular, for $\rho > 0$ to be specified later, we choose $\rho_h$ small enough such that
    \[
    \|u_h\|_{C^{2m + 2}([0,T] \times \bar{\Omega})} + \|g_0\|_{H^{2m+1}}+ \|g_1\|_{H^{2m-1}}\leq \rho.
    \]
    Now we consider $w = u - u_h$ and we rewrite the nonlinear term as
    \[
    \beta\ddt (u^2)
    = 2\beta(w + u_h)\ddt w
    + 2 \beta( w + u_h)\dt^2 u_h
    + 2 \beta (\dt (w + u_h))^2.
    \]
    Then $w$ solves the initial value problem 
    \begin{equation}\label{eq_w}
        \begin{aligned}
            (1-2\beta(w + u_h))\ddt w - \Laplace w - \dt \Laplace w  + q(t,x)w &= F(u_h, w), & \  & \mbox{on } (0,T)\times \Omega,\\
            w(t,x) &= 0, & \ &\mbox{for  } x\in \partial \Omega,\\
            w = \tg_0(x),  \quad  \partial_t u &= \tg_1(x), & \ &\mbox{for  } t=0,
        \end{aligned}
    \end{equation}
    where we write
    \begin{align*}
        &\tg_0(x) = g_0(x) -  u_h(0), \qquad \tg_1(x) = g_1(x)- \dt u_h(0),\\
        &F(u_h, w) = -Pu_h
        + 2\beta (\dt (w + u_h))^2
        + 2 \beta( w + u_h)\dt^2 u_h.
    \end{align*}
    Further, we define $\tg_l = g_l - \dt^l u_h(0)$, for $l = 2, \ldots, m$.
    With $(g_0, g_1, h)\in \Gm$, we have \[
    \tg_l \in H_0^1(\Omega) \cap H^{2(m-l)+1}(\Omega).
    \]

    In the following, we choose fixed $r_0 > 0$ to satisfy Proposition \ref{pp_energy}.
    For $r < \epv/2$ to be specified later, we define the set $\tzm(r, T)$, which containing all the functions $w$ such that
    \begin{align*}
        &w \in \bigcap_{k=0}^{m} C^{m-k}([0,T]; H^{2k+1}(\Omega)), \quad \text{with }
        \|w\|_\zm \leq r \text{ and }  \dt^l w(0) = \tg_l, \text{ for } l = 0, \ldots, m.
    \end{align*}
    We consider the linearized problem, i.e., for fixed $v \in \tzm(r, T)$, we consider
    \begin{equation}\label{eq_linear_w}
        \begin{aligned}
            (1-2\beta (v + u_h)) \ddt w - \Delta w - \dt \Delta w + q(t,x)w  &= F(u_h, v), & \  & \mbox{on } (0,T)\times \Omega,\\
            w(t,x) &= 0, & \ &\mbox{for  } x\in \partial \Omega,\\
            w = \tg_0(x),  \quad  \partial_t w & = \tg_1(x), & \ &\mbox{for  } t=0.
        \end{aligned}
    \end{equation}
    Let ${J}$ be the operator that maps $v$ to the solution $w$.
    Then with $u_h \in C^{2m+2}([0, T] \times \bar{\Omega})$ and $v \in \zm(r,T)$, we have
    \begin{align*}
        &\|F\|_\zmm \leq C(\|u_h\|_{C^{2m+2}([0, T] \times \bar{\Omega})} + (\|v\|_\zmm + \|u_h\|_{C^{2m+2}([0, T] \times \bar{\Omega})})^2),
    \end{align*}
    where we use \cite[Claim 1]{Uhlmann2021a}.
    According to Lemma \ref{lm_compt} in {Appendix} \ref{sec_apdx}, the data $\tg_0, \tg_1, F$ satisfy the $m$th-order compatibility condition (\ref{def_comp_linear}) for the linearized equation (\ref{eq_linear_w}).
    Then by Proposition \ref{pp_energy},
    there exists a unique solution
    \[
    w \in \bigcap_{k=0}^{m} C^{m-k}([0,T]; H^{2k+1}(\Omega))
    \]
    such that
    \begin{align}\label{eq_fandu}
        \| w \|_{Z^m}
        \leq C (\|u_h\|_{C^{2m+2}([0, T] \times \bar{\Omega})} & + \|g_0\|_{H^{2m+1}}+ \|g_1\|_{H^{2m-1}}\\
        & +(\|v\|_\zmm + \|u_h\|_{C^{2m+2}([0, T] \times \bar{\Omega})})^2) ,\nonumber
    \end{align}
    for a new constant $C$ independent of $r$ and $\rho$.
    For
    $r < \min\{\epv/2, 1, {1}/{(5C)}\}$ and $\rho = r^2$,
    the above inequality implies
    \[
    \| w \|_{Z^m} \leq C(\rho + (r + \rho)^2) \leq 5C r^2 \leq r.
    \]
    Thus  ${J}$ maps $\tzm(r, T)$ to itself.

    In the following, we prove ${J}$ is a contraction for sufficiently small $r$.
    Indeed, for $w_j = {J}(v_j)$ with $v_j \in \tzm(r, T)$, we write $w_0 = u_2 - u_1$ and it satisfies
    \begin{align*}
        &(1 - 2\beta (v_1 + u_h)) \ddt w_0  - \Laplace w_0- \dt (\Laplace)  w_0 + q(t,x) w_0\\
        =&  2\beta(v_2 - v_1) \ddt w_2 + 2 \beta (\dt v_2 - \dt v_1)(\dt v_2 + \dt v_1 + 2\dt u_h) + 2\beta(v_2 - v_1) \dt^2 u_h.
    \end{align*}
    We denote the right-hand side by ${I}$ and compute
    \begin{align*}
        \|{I}\|_\zmm &\leq C' (\|w_2\|_\zm + \|v_1\|_\zm + \| v_2\|_\zm + \|u_h\|_{C^{2m+2}([0, T] \times \bar{\Omega})})\|v_1 - v_2\|_\zm,
    \end{align*}
    for some positive constant $C'$ independent of $r$ and $\rho$,
    again by \cite[Claim 1]{Uhlmann2021a}.
    Note that previously we choose $r$ small enough such that $\rho = r^2 \leq  r$ and $\|u_j\|_\zm \leq r$, for $j =1,2$.
    Then we have
    \[
    \|{I}\|_\zmm \leq 4C'r \|v_1 - v_2\|_\zm.
    \]
    By Proposition \ref{pp_energy},
    one obtains
    \begin{align*}
        &\| w_2 - w_1 \|_\zm \leq C \|{I}\|_\zmm
        \leq 4CC'r \|v_1 - v_2\|_\zm.
    \end{align*}
    If we additionally assume $r < 1/(4CC')$, then
    \[
    \|{J}(v_2 -v_1)\|_\zm < \|v_2 - v_1\|_\zm,
    \]
    which proves that ${J}$ is a contraction.
    Thus, there exists a unique solution $\widetilde{w}$ as the fixed point of $J$ in $\tzm(r, T)$, which is the solution to the nonlinear problem (\ref{eq_w}).
    Moreover,
    if we choose $r$ small enough such that
    $C r \leq 1/4$ in (\ref{eq_fandu}), then we have
    \[
    \|\tilde{w}\|_\zm \leq 2C (\|u_h\|_{C^{2m+2}([0, T] \times \bar{\Omega})} + \|g_0\|_{H^{2m+1}}+ \|g_1\|_{H^{2m-1}}).
    \]
    This implies there exists a unique solution $\tilde{u}$ to the  initial-boundary value problem (\ref{eq_nonlinear_bvp}) in $\tzm$ satisfying
    \[
    \|\tilde{u}\|_\zm \leq C (\|u_h\|_{C^{2m+2}([0, T] \times \bar{\Omega})} + \|g_0\|_{H^{2m+1}}+ \|g_1\|_{H^{2m-1}}),
    \]
    where $C$ is a new constant depending on $T$, $\beta$, $q$, and the domain $\Omega$.
\end{proof}

\section{Preliminaries}\label{sec_prelim}
In this section, for convenience, we use $x$ to denote a point in $\Rn$
and $P(D)$ to denote a differential operator, where $D = -{i}\partial_x$.
Let $\alpha = (\alpha_1, \ldots, \alpha_n)$ be a multi-index and we write
$
\partial^\alpha = \partial_1^{\alpha_1} \ldots\partial_n^{\alpha_n}.
$
To apply these results to our model, we actually abuse the notation by considering $(t,x) \in \mathbb{R} \times \Rn$ and a differential operator $P(D) = P(D_t, D_x)$.
\subsection{The spaces $\Bpk$}
Let $H^s(\mathbb{R}^n)$ be the Sobolev space.
To describe the regularity of fundamental solutions, we introduce spaces of distributions which generalize the $H^s$ spaces, for more details see \cite[Section 10.1]{H2005analysis}.
We briefly recall some definitions and results in the following.

Let $k$ be a positive function in $\mathbb{R}^n$.
We say $k$ is a \textbf{temperate weight function} if there exist positive constants $C$ and $N$ such that
\[
k(\xi + \eta) \leq (1+ C|\xi|)^N k(\eta), \quad \xi, \eta \in \mathbb{R}^n.
\]
For example, the function $k_s(\xi) = (1+ |\xi|^2)^{{s}/{2}}$ is a temperate weight that corresponds to $H^s(\Rn)$.
Another example is the function $\tP$ define by
\[
\tP(\xi)^2 = \sum_{|\alpha| \geq 0 } |P^{(\alpha)}(\xi)|^2, \quad \xi \in \Rn
\]
where $P(\xi)$ is a polynomial
and we write $P^{(\alpha)}(\xi) \coloneqq \partial^\alpha P(\xi)$.
One can show by Taylor expansion that there is a positive constant $C$ only depending on $\deg P$ and $n$ such that
\begin{align}\label{eq_tP}
\tP(\xi + \eta) \leq C(1+ |\xi|)^m \tP(\eta).
\end{align}
Thus, it is a temperate weight function.
\begin{df}[{\cite[Definition 10.1.6]{H2005analysis}}]
For a temperate weight function $k(\xi)$ and $1 \leq p \leq \infty$, we denote by $B_{p, k}(\mathbb{R}^n)$ the set of all distributions $u \in \mathscr{S}'$ such that $\hat{u}$ is a function and
\begin{align}\label{def_Bpk}
\|u \|_{p,k} = ((2\pi)^{-n} \int |k(\xi)\hat{u}(\xi)|^p \diff \xi)^{1/p} \leq \infty.
\end{align}
When $p = \infty$, we define $\|u \|_{\infty,k} = \mathrm{ess.\ sup} |k(\xi)\hat{u}(\xi)|$.
\end{df}
One can show $B_{p, k}(\mathbb{R}^n)$ is a Banach space with the norm (\ref{def_Bpk}).  Moreover, we have
$
\mathscr{S}(\mathbb{R}^n) \subset B_{p, k}(\mathbb{R}^n) \subset \mathscr{S}'(\mathbb{R}^n),
$
with $C_0^\infty(\mathbb{R}^n)$ dense in $\Bpk(\mathbb{R}^n)$ for $p < \infty$.
When $p = 2$, one has
\[
\|u \|_{2,\tP} = (\sum_{|\alpha| \geq 0} \| P^{(\alpha)}(D) u\|_{L^2(\mathbb{R}^n)}^2)^{1/2} .
\]
In particular, we have $\|u \|_{2,1} = \|u\|_{\ltwo(\Rn)}$.
Additionally, we have the following properties.
\begin{pp}[{\cite[Theorem 10.1.12]{H2005analysis}}]\label{pp_convl}
If $u_1 \in B_{p, k_1}(\mathbb{R}^n) \cap \mathcal{E}'(\mathbb{R}^n)$ and $u_2 \in B_{\infty, k_2}(\mathbb{R}^n)$, it follows that $u_1 \ast u_2 \in B_{p, k_1k_2}(\mathbb{R}^n)$ and we have the estimate
\begin{align*}
    \| u_1 \ast u_2 \|_{p, k_1k_2} \leq \|u_1\|_{p, k_1} \|u_2\|_{\infty, k_2}.
\end{align*}
\end{pp}
\begin{pp}[{\cite[Theorem 10.1.15]{H2005analysis}}]\label{pp_psiu}
    If $u \in B_{p, k}(\mathbb{R}^n)$ and $\psi \in \mathscr{S}(\mathbb{R}^n)$, it follows that $\psi u \in B_{p, k}(\mathbb{R}^n)$ with
    \begin{align*}
        \| \psi u \|_{p, k} \leq \|\psi\|_{1, M_k} \|u\|_{p, k},
    \end{align*}
where $M_k(\xi)  = \sup_{\eta\in \mathbb{R}^n} \frac{k(\xi + \eta)}{k(\eta)}$ is a temperate weight function induced by $k$.
\end{pp}
Next, let $X$ be an open subset of $\mathbb{R}^n$ and we define $\Bpkloc(X)$ as local spaces containing $u \in \mathcal{D}'(X)$ such that for every $x_0 \in X$, there exists $\phi \in C_0^\infty(X)$ with $\phi(x_0) \neq 0$ and $\phi u \in \Bpk(\mathbb{R}^n)$.
With this notation, we have the following result for fundamental solutions to differential operators with constant coefficients.


\begin{pp}[{\cite[Theorem 10.2.1]{H2005analysis}}]\label{pp_E0}
Let $P(D)$ be a partial differential operator with constant coefficients which is not equal to $0$. Then there is a fundamental solution $E_0 \in \BinPloc(\mathbb{R}^n)$ such that $PE_0 = \delta$, where $\delta$ is the Dirac measure at $0$.
In particular, we have
\begin{align}\label{eq_E0norm}
\|E_0/(\cosh|x|)\|_{\infty, \tP} \leq C,
\end{align}
for a positive constant $C$ only depending on $\deg P$ and $n$.
\end{pp}
Furthermore, we prove the following proposition, by modifying
{\cite[Theorem 1.2]{isakov1991completeness}} based on
{\cite[Theorem 10.3.7]{H2005analysis}}.
\begin{pp}\label{pp_Eestimate}
Let $X$ be a bounded open set in $\mathbb{R}^n$ with smooth boundary.
There exists a bounded linear operator $E$  such that
\[
PE f = f, \quad \text{ for any } f \in H^s(X), \text{ where } s\geq0 \text{ is an integer,}
\]
and for any differential operator $Q$ with constant coefficients we have
\begin{align}\label{eq_Eestimate}
    \|Q(D) Ef\|_{H^s(X)} \leq C \sup_{\xi \in \mathbb{R}^n} \frac{\widetilde{Q}(\xi)}{\tP(\xi)} \|f\|_{H^s(X)},
\end{align}
where $C$ depends only on $\deg P$, $n$, and $X$.
\end{pp}
\begin{proof}
When $s \geq 0$ is an integer, recall $H^s(X)$ consist of all $f \in L^2(X)$ such that $\partial^\alpha u \in L^2(X)$, for any multi-index $\alpha = (\alpha_1, \ldots, \alpha_n)$ with $|\alpha| \leq s$.
Moreover, for $f \in H^s(X)$,
there exists $f_0 \in H^s(\mathbb{R}^n)$ such that $\|f_0\|_{H^s(\mathbb{R}^n)} \leq C \|f\|_{H^s(X)}$ with $f_0 = f$ in $X$.

Let $Ef$ be the restriction of $u_0 = E_0 \ast f_0$ to $X$, where $E_0$ is a fundamental solution given by Proposition \ref{pp_E0}.
It follows that $PEf = f$.
To prove that $Q(D)E$ is bounded, we consider a smooth cutoff function $\psi \in C_0^\infty(X)$ with $\psi = 1$ in a neighborhood of the closure of $X- X = \{x-y: x, y \in X\}$.
We set $F_0 = \psi E_0$ and $g(x) = \cosh(x)$.
Then by Proposition \ref{pp_psiu}, we have $F_0 \in \BinP(\Rn)$ with
\begin{align*}
\|F\|_{\infty, \tP}  = \|(\psi g) (E_0/g)\|_{\infty, \tP}  \leq
\|\psi g\|_{1, M_{\tP}} \|E_0/g\|_{\infty, \tP}.
\end{align*}
Here by (\ref{eq_tP}) there exists a constant $C> 0$ only depending on $\deg P$ and $n$ such that
\[
M_{\tP}(\xi)  = \sup_{\eta\in \mathbb{R}^n} \frac{\tP(\xi + \eta)}{\tP(\eta)} \leq C(1+ |\xi|)^m,
\]
and therefore $\|\psi g\|_{1, M_{\tP}}$ is bounded.
Combining it with (\ref{eq_E0norm}), we have
$\|F\|_{\infty, \tP}  \leq C$, for a new constant $C$ only depending on $\psi$, $\deg P$, $n$, $X$.
Hence, for any multi-index $\alpha$ with $|\alpha| \leq s$, we have
\begin{align*}
\| \partial^\alpha (Q(D) Ef)\|_{L^2(X)} = \|  Q(D) \partial^\alpha (Ef)\|_{L^2(X)} \leq \| Q(D) \partial^\alpha (F_0 \ast f_0)\|_{2, 1}.
\end{align*}
Note $F_0$ has compact support and therefore $Q(D)\partial^\alpha (F_0 \ast f_0) = (Q(D)F_0) \ast \partial^\alpha f_0$.
This implies
\begin{align*}
    \| \partial^\alpha (Q(D) Ef)\|_{L^2(X)}
    &\leq \| Q(D)  (F_0 \ast \partial^\alpha f_0)\|_{2, 1}
    \leq \|\partial^\alpha f_0\|_{2,1} \|Q(D) F_0\|_{\infty, 1}\\
    &\leq C \|f\|_{H^s(X)} \|Q(D) F_0\|_{\infty, 1}
    \leq C \|f\|_{H^s(X)} \| F_0\|_{\infty, \tP}\sup_{\xi \in \mathbb{R}^n} \frac{\widetilde{Q}(\xi)}{\tP(\xi)},
\end{align*}
where we use Proposition \ref{pp_convl} and the last inequality comes from the estimate
\[
|\widehat{Q(D)F_0}(\xi)|  = |Q(\xi) \widehat{F}_0(\xi)| \leq
\sup_{\xi \in \mathbb{R}^n} \frac{\widetilde{Q}(\xi)}{\tP(\xi)}
| \tP(\xi)\widehat{F}_0(\xi)|.
\]
Thus, we have the desired estimate. 
\end{proof}

\subsection{The characteristic Cauchy problem}\label{subsec_cauchy}
In Section \ref{sec_gocnew}, we would like to show the reminder term is relatively small with respect to a large parameter, and moreover it is supported in { a half space $\{t \geq \vep\}$}, for some small number $\vep > 0$.
This enables us to obtain solutions satisfying the zero initial conditions.
For this purpose, we recall some notations and results in  \cite[Section 12.8]{H2005analysis}.

In the following, we use the notation $\check{u}(x) = u(-x)$.
By \cite[Theorem 7.1.10]{Hoermander2003}, for any $u \in \mathscr{S}'(\Rn)$, one has
\[
\doublehat{u} = (2\pi)^n \check{u},
\]
where $\hat{u} = \int e^{-i x\cdot \xi} f(x) \diff x$ is the Fourier transform.
\begin{pp}[{\cite[Theorem 10.1.14]{H2005analysis}}]\label{pp_10114}
    If $L$ is a continuous linear form on $B_{p,k}(\Rn)$ for $p < \infty$, then we have
    \[
    L(u) = \check{v}(u), \qquad u \in \mathscr{S}(\Rn),
    \]
    for some $v \in B_{p', 1/k}(\Rn)$ with ${1}/{p} + {1}/{p'} = 1$.
    The norm of this linear form is $\|v\|_{p', 1/k}$. Hence $B_{p', 1/k}(\Rn)$ is the dual space of $B_{p,k}$ and the canonical bilinear form in $B_{p,k}(\Rn) \times B_{p', 1/k}(\Rn)$ is the continuous extension of the bilinear form $\check{v}(u)$, where $v \in B_{p', 1/k}(\Rn)$ and $u \in \mathscr{S}(\Rn)$.
\end{pp}

Recall $P(D)$ is a partial differential operator
with constant coefficients in $\Rn$.
Let $N = (0,  \ldots,0, 1) \in \Rn$ and for convenience we consider the half space \[
H = \{x: x_n \geq 0\} = \{x: x \cdot N \geq 0\}.
\]
We study the Cauchy problem for $P(D)$ in  $H$ with vanishing data on the boundary $\partial H$.
We suppose $\partial H$ is characteristic with respect to $P(D)$, i.e., $P_m(N) = 0$, where $P_m$ is the principal part.
By \cite[Theorem 8.6.7]{Hoermander2003}, there is no uniqueness for the Cauchy problem $(P+q) u = f$ with $f$ supported in $H$, unless growth conditions are imposed.
Our aim is to find a solution supported in $H$ with the desired estimate, following the proof of H\"{o}rmander.
First, we recall the following proposition.
\begin{pp}[{\cite[Theorem 12.8.1]{H2005analysis}}]\label{pp_1281}
    Let $P(D)$, $N$, and $H$ be defined as above.
    Then the following statements are equivalent.
    \begin{itemize}
        \item[(a)] $P(D)$ has a fundamental solution with support in $H$.
        \item[(b)]\phantomsection\label{condition_b} There exist constant $A_1$ and $A_2$ with $A_1 > 0$ such that for every solution $\sigma(\zeta)$ of $P(\zeta + \sigma N) = 0$ which is analytic and single-valued in a ball $B$ with real center and radius $A_1$, we have
        \[
        \sup_B \mathrm{Im} \sigma(\zeta) \geq A_2.
        \]
        \item[(c)] If $ 1 \leq p < \infty$ and $k$ is a temperate weight function, the equation $P(D)u = f$ has a solution $u \in \Bpkloc(\Rn)$ with support in $H$, for every $f \in \Bpkloc(\Rn)$ with support in $H$.
    \end{itemize}
\end{pp}
We emphasize that condition (\refb)  is invariant under translation, i.e., it holds for $P(D)$ exactly when it holds for $P(D + \zeta^0)$, where $\zeta^0 \in \mathbb{C}^n$.
However, the constant $A_1, A_2$ are changed by $\zeta^0$ correspondingly.

Next, we define the following quotient norm
\[
\|u\|_{p,k}^- = \inf_{v}\|v\|_{p,k}, \quad \text{where } u, v \in \mathscr{S}(\Rn) \text{ with $u = v$ in  $-H$},
\]
where we define $-H  = \{x: -x \in H \}$. 
Note this norm allows us to focus on the restrictions of $u$ to $-H$.
We have the following proposition.
\begin{pp}[{\cite[Theorem 12.8.12]{H2005analysis}}]\label{pp_12812}
    Assume $P$ satisfies condition (\refb) in Proposition
    \ref{pp_1281} with $A_2 = A_1 +1$.
    Then one can find constants $\kappa$ and $C$ for arbitrary $p \in [1, \infty]$ and any temperate weight function $k$  such that
    \[
    \| v\|_{p,k}^- \leq C e^{\kappa M}  \|P(D) v\|_{p, k/\tP}^-,
    \]
    when $v \in C_0^\infty(\Rn)$ and we assume $\supp v \cap (-H) \subset \{(x',x_n): |x'| \leq M\}$.
\end{pp}
According to the remarks in the proof of
\cite[Theorem 12.8.17]{H2005analysis},
this constant $C e^{\kappa M}$ only depends on $A_1, A_2$, the dimension $n$, and $\deg P$.
Indeed, \cite[Theorem 12.8.17]{H2005analysis} is proved based on Proposition 12.8.4, which depends on Lemma 12.8.5 to 12.8.10 there.
In particular, the set $F_m(Z)$ defined in \cite[Lemma 12.8.7]{H2005analysis} and the estimate there are independent of the choice of the polynomial $R$.

Thus, one can show a local existence result using estimates from Proposition \ref{pp_12812}, based on the proof of \cite[Theorem 12.8.3]{H2005analysis}.

\begin{pp}\label{pp_12813}
    Let $1 < p \leq \infty$, $k$ be a temperate weight function on $\Rn$, and $X \subset \Rn$ be a bounded open set.
    Moreover, suppose $X \subset \{(x', x_n): |x'| \leq M\}$ for some constant $M$.
    Assume $P(D)$ satisfies condition (\refb) with $A_1 = A_2 +1$.
    If $f \in \Bpk(\Rn) \cap \mathcal{E}'(\Rn)$ and $\supp f \subset H$, then one can find $u \in B_{p, k \tP }(\Rn)$ such that
    $\supp u \subset H$ and $P(D) u = f$ in $X$ with
    \[
    \| u\|_{p,k\tP} \leq C  \|f\|_{p, k},
    \]
    where $C$ depends on $n$, $X$, $A_1, A_2$, and $\deg P$.
\end{pp}
\begin{proof}
    Recall we write $\check{u} = u(-x)$ and let $\Pad$ be the formal adjoint of $P$.
    Note that we have $P(D) u = f$ in $X$
    if and only if  $\Pad (D) \check{u} = \check{f}$ in $-X$,
    which means
    \[
    \check{u} (P(D) v) = \check{f}(v), \qquad \text{for any } v \in C_0^\infty(-X).
    \]
    By Proposition \ref{pp_10114}, with $f \in \Bpk(\Rn)$, it defines a continuous linear form $\check{f}(w)$ for $w \in \mathscr{S}(\Rn)$.
    This linear form can be extended to $B_{p', 1/k}(\Rn)$, when $p'$ satisfies  ${1}/{p}+ {1}/{p'} = 1$.
    The norm of this linear form is given by $\|f\|_{p,k}$.
    Then we have
    \[
    |\check{f}(v)| \leq \|f\|_{p,k} \|v\|_{p', 1/k}, \qquad v \in B_{p', 1/k}(\Rn). 
    \]
    We emphasize that the notations $p$ and $p'$ are switched, compared to Proposition \ref{pp_10114}.
    Thus, we assume $p > 1$ for convenience.

    Note that $-X \subset  \{(x', x_n): |x'| \leq M\}$ as well.
    According to Proposition \ref{pp_12812}, there exist constants $\kappa$ and
    $C'$ depending on $A_1, A_2$, the dimension $n$, and $\deg P$, such that
    \[
    \|f\|_{p,k} \|v\|_{p', 1/k} \leq C' e^{\kappa M}\|f\|_{p,k} \|P(D)v\|^-_{p', \tP/{k}},
    \]
    for any $v \in C_0^\infty(-X)$.
    This implies that we can regard $\check{u}: P(D)v \rightarrow \check{f}(v)$ as a continuous linear form on $P(D)(C_0^\infty(-X))$.
    By Hahn-Banach theorem, there exists a continuous linear form $\check{U}$ on $C_0^\infty(\Rn)$, such that
    \[
    |\check{U}(w)| \leq C \|w\|^-_{p', \tP/{k}}, \qquad \text{for any } w \in C_0^\infty(\Rn),
    \]
    where we write $C = C' e^{\kappa M} \|f\|_{p,k}$.
    We must have $\check{U}$ supported in $-H$, which implies $U$ is supported in $H$.
    Again by Proposition \ref{pp_10114},
    we have $U \in B_{p, k\tP}(\Rn)$ with
    \[
    \| u\|_{p,k\tP} \leq C  \|f\|_{p, k}.
    \]
\end{proof}
\begin{crl}\label{crl_fundamental}
    Let $X, P(D)$ be as in Proposition \ref{pp_12813}.
    In particular, there exists a solution $\hEz \in B_{\infty,\tP}(\Rn)$ with
    $\supp \hEz \subset H$ such that $P(D) \hEz = \delta$ in $X$ and
    \[
    \| \hEz \|_{\infty,\tP} \leq C,
    \]
    where $C$ depends on $n$, $X$, $A_1, A_2$, and $\deg P$.
\end{crl}
\begin{proof}
Note that $\delta \in B_{\infty,1}(\Rn) \cap \mathcal{E}'(\Rn)$ with $\supp \delta \subset H$ and $\|\delta\|_{\infty, 1} = 1$.
By Proposition \ref{pp_12813}, we have the desired estimate for $E$.
\end{proof}
We emphasize such $\hEz$ is almost a fundamental solution, if we choose $X$ large enough.
More precisely, we prove the following proposition.
\begin{pp}\label{pp_Eestimate_new}
    Let $X, P(D)$ be as in Proposition \ref{pp_12813}.
    Let $\tX$ be a small open neighborhood of $X$ and $X_0$ be a small open neighborhood of $\tX-\tX$.
    Then there exists a bounded linear operator $\hE$  such that
    \[
    P\hE f = f \quad \text{in } X, \quad \text{ for any } f \in H^{s}(\tX) \text{ and } \supp f \subset H ,
    \]
    with  $\supp \hE f \subset H$ and 
    \begin{align}\label{eq_Eestimate_H}
        \| \hE f\|_{H^s(\tX)} \leq C \sup_{\xi \in \mathbb{R}^n} \frac{1}{\tP(\xi)} \|f\|_{H^s(\tX)},
    \end{align}
where $C$ depends on $X_0$, $A_1, A_2$, the dimension $n$, and $\deg P$.
\end{pp}
\begin{proof}
We are motivated by the proof of \cite[Theorem 12.8.14]{H2005analysis}.
Note $X_0$ as a small neighborhood of $\tX - \tX$ is also bounded.
By Corollary \ref{crl_fundamental}, there exists $\hEz$ with $\supp \hEz \subset H$ such that
$P(D) \hEz = \delta$ in $X_0$.
Thus, we set $g = P(D) \hEz - \delta$ and $g$ vanishes in $X_0$.


Now let $\psi \in C_0^\infty(\tX)$ be a smooth cutoff function with $\psi = 1$ in the closure of $X$.
Since $f \in H^s(\tX)$ with $\supp f \subset H$, then $\tf = \psi f$ is supported in $H$ satisfying
\[
\|\tf\|_{H^s(\Rn)} \leq C \|f\|_{H^s(\tX)}, \qquad \tf = f \text{ in } X.
\]
Then we set $u = \hEz \ast \tf$ to have
\[
P(D)u = (P(D)\hEz) \ast \tf = (\delta + g) \ast \tf = \tf + g \ast \tf.
\]
Recall $g$ is supported in $\Rn \setminus X_0$ and $\tf$ is supported in $\tX$.
Since $(\Rn \setminus X_0) \cap (\tX - \tX) = \emptyset$, we have
\[
g \ast \tf(x) = \int g(x-y)\tf(y) \diff y= 0, \qquad \text{for any } x \in \tX.
\]
This implies $P(D) u = f$ in $X$. Since $\supp \hEz \subset H$ and $\supp \tf \subset H$, we have $\supp u \subset H$. Moreover, we estimate
\begin{align*}
    \|u \|_{p, k}  = \|\hEz \ast \tf\|_{p,k}\leq \|\hEz\|_{\infty, 1}\|\tf\|_{p,k}
    \leq  \sup_{\xi \in \mathbb{R}^n} \frac{1}{\tP(\xi)} \|\hEz\|_{\infty, \tP}\|\tf\|_{p,k}\leq C\sup_{\xi \in \mathbb{R}^n} \frac{1}{\tP(\xi)} \|\tf\|_{p,k},
\end{align*}
where we use Proposition \ref{pp_convl} and the second inequality comes from the estimate
\[
|\widehat{\hEz}(\xi)|  \leq
\sup_{\xi \in \mathbb{R}^n} \frac{1}{\tP(\xi)} |\tP(\xi) \widehat{\hEz}(\xi)|.
\]
Here the constant $C$ depends on the dimension $n$, $X_0$, $A_1, A_2$, and $\deg P$.
In particular,
we can choose $p = 2$ and $k = k_s = {(1 + |\xi|^2)^{s/2}}$ to have
\[
\| u\|_{H^s(\Rn)}  = \|\tu \|_{2,k_s}
\leq C \sup_{\xi \in \mathbb{R}^n} \frac{1}{\tP(\xi)} \|\tf\|_{2,k_s}
\leq C \sup_{\xi \in \mathbb{R}^n} \frac{1}{\tP(\xi)} \|f\|_{H^s(\tX)}.
\]
Thus, we define $\hE f = \psi u$ to have the desired result.
\end{proof}
\section{Geometric optics constructions}\label{sec_goc}
In this section, we would like to construct an approximate solution to the linear problem
\begin{align}\label{eq_linear1}
    \begin{aligned}
        \opq u  &= 0, & \  & \mbox{on } (0, T) \times \Omega, \\
    \end{aligned}
\end{align}
For a large parameter $\rho$ and fixed $\omega \in S^{n-1}$, we consider an asymptotic solution of the form
\begin{align}\label{eq_vN}
    \vN(t,x) = e^{i\varphi(t,x,\rho, \omega)} a(t,x, \rho, \omega),
\end{align}
where $\varphi$ is the phase function
and $a$ is the amplitude satisfying
\[
a(t,x, \rho, \omega) = \sum_{j=0}^N{\rho^{-j}}{a_j(t,x, \omega)}.
\]
Plugging this form into the linear problem, we have
\begin{align*}
    &\opq(e^{i\varphi} a ) \\
    = &e^{i\varphi}(
    -(\partial_t \varphi)^2 a + i \partial_t^2 \varphi a
    + 2i \partial_t \varphi \partial_t a
    + \partial_t^2 a
    -(-(\nabla \varphi \cdot \nabla \varphi) a + i \Delta \varphi a +  2 i\nabla \varphi \cdot \nabla a + \Delta a)\\
    &-(-i \partial_t \varphi(\nabla \varphi \cdot \nabla \varphi) a
    - 2(\partial_t \nabla \varphi \cdot \nabla \varphi)a
    - (\nabla \varphi \cdot \nabla \varphi)  \partial_t a
     - \partial_t \varphi \Delta \varphi a
    + i \partial_t \Delta \varphi a
    + i\Delta \varphi \partial_t a \\
    &- 2 \partial_t \varphi \nabla \varphi \cdot \nabla a
    + 2i (\partial_t \nabla \varphi) \cdot \nabla a
    + 2i \nabla \varphi \cdot \nabla \partial_t a
    + i \partial_t \varphi \Delta a + \partial_t \Delta a
    ) + q a
    ).
\end{align*}
We choose the Eikonal equation
\[
-(\partial_t \varphi)^2 +i \partial_t \varphi(\nabla \varphi \cdot \nabla \varphi)   = 0,
\]
which is satisfied by a linear phase function
\[
\varphi(t,x, \rho, \omega) = -i\rho^2 t + i\rho(x \cdot \omega).
\]
This phase function has complex frequency and gives us exponentially growing solutions and we have
\begin{align*}
    &\opq(e^{\rho^2 t-\rho(x \cdot \omega)} a )
    = e^{\rho^2 t-\rho(x \cdot \omega)}(
    \rho^3 \Tzero a
    + \rho^2 \Tone a
    + \rho \Ttwo a
    + (P+q) a),
\end{align*}
where we introduce the following operators
\begin{align*}
    \Tzero = 2 \omega \cdot \nabla, \quad
    \Tone  = \partial_t  -  \Delta  - 1, \quad
    \Ttwo = 2 ( \omega \cdot \nabla  + \omega \cdot \nabla \partial_t), \quad
    P = \op.
\end{align*}
It follows that the amplitude satisfies
\begin{align}
    & \Tzero a_0(t, x, \omega) = 0, \label{transport_0}\\
    &\Tzero a_{1}(t, x, \omega) = -\Tone a_0, \label{transport_1}\\
    &\Tzero a_{2}(t, x, \omega) = -\Tone a_1 - \Ttwo a_0,
    \label{transport_2}\\
    &\Tzero a_{j}(t, x, \omega) = -\Tone a_{j-1} - \Ttwo a_{j-2} - (P+q) a_{j-3}, \quad \text{ for } k \geq 3. \label{transport_j}
\end{align}

To find the amplitude  $a(t, x, \rho, \omega)$, one solves $a_0$ from (\ref{transport_0}) first and then plug it in (\ref{transport_1}) to solve $a_1$, and then plug $a_0, a_1$ in
(\ref{transport_2}) to solve $a_2$, and next solve $a_j$ from (\ref{transport_j}) for $j \geq 3$.
One possible solution to (\ref{transport_0}) is
\begin{align}\label{eq_a0}
    a_0(t, x,\omega) = \phi(t) \prod_{j = 2}^{n} \chi(\omega_j \cdot (x - y_0)),
\end{align}
where $\phi(t) \in C_0^\infty(\mathbb{R})$ is supported in $(0, T)$
and $\chi \in C_0^\infty(\mathbb{R})$ is supported in $(-\ep, \ep)$ with  $\chi(s) = 1$ near $0$, for some $\ep > 0$.
In addition, here we pick a fixed point $y_0 \in \partial \Omega$ and suppose $\omega_1, \ldots, \omega_{n}$ form an orthogonal basis in $\mathbb{R}^n$ with $\omega_1 = \omega$.
Note that for each fixed $t$, the leading term $a_0(t,x, \omega)$ is a smooth function supported in a small $\ep$-neighborhood of the ray $x = s \omega + y_0$.
This amplitude is used in \cite{lassas2022inverse} for a time-dependent Schr\"odinger equation and related inverse problems.

Then we define the set
\[
\Sigma({y_0, \omega}) = \{y \in \mathbb{R}^n: \omega \cdot (y-y_0) = 0\}.
\]
For any $x \in \Rn$, there exists a unique $y \in \Sigma({y_0, \omega})$ such that
$x = s \omega + y$ for some $s \in \mathbb{R}$.
With vanishing initial conditions on $\Sigma({y_0, \omega})$,
we have
\begin{align*}
    & a_{1}(t, s\omega + y, \omega) = -\frac{1}{2}\int_0^s \Tone a_0(s' \omega + y) \diff s',\\
    &a_{2}(t, s\omega + y, \omega) =-\frac{1}{2} \int_0^s (\Tone a_1+ \Ttwo a_0)(s' \omega + y) \diff s',\\
    &a_{j}(t, s\omega + y, \omega) = -\frac{1}{2}\int_0^s (\Tone a_{j-1} + \Ttwo a_{j-2} + (P+q) a_{j-3})(s' \omega + y) \diff s', \quad \text{ for } k \geq 3.
\end{align*}
Note that $a_0(t, x, \omega)$ is supported with respect to $t$ in $(0, T)$ and $\Tzero$ is independent of $t$.
This implies $a_j(t, x)$ as well as $v_N(t, x)$
have the same support with respect to $t$.
In addition, the construction above implies that
\begin{align}\label{eq_errorterm}
    (P+ q) \vN &= -\rho^{-N+2}e^{i\varphi} ((T_1 a_N + T_2 a_{N-1} + (P+q) a_{N-2}) \\
    &\quad \quad \quad  - \rho^{-1} (T_2 a_N + (P+q) a_{N-1}) - \rho^{-2}(P+q) a_N) \coloneqq \rho^{-N+2} e^{i\varphi} F_N \nonumber,
\end{align}
for $N \geq 2$. When $N=0$, only $a_N$ is involved and when $N = 1$, only $a_N, a_{N-1}$ are involved.
Note that for  $\rho \gg 1$, we have
$
\|F_N \|_{H^s((0, T) \times \Omega)}
\leq C_{F_N},
$
where $C_{F_N}$ depends on $q, n, N, \Omega$ and the choice of $\phi, \chi$.
\subsection{The backward problem}\label{subsec_backward_goc}
Next, we consider the backward strongly damped wave equation $\opqstar w  = 0$.
Similarly, for a large parameter $\rho$  and fixed $\varpi \in S^{n-1}$,
we construct an approximate solution by considering
\begin{align}\label{eq_wN}
    w_N(t,x) = e^{i\psi(t,x,\rho, \varpi)} b(t,x, \rho, \varpi),
\end{align}
where $\psi$ is a linear phase function given by
\[
\psi(t,x,\rho, \varpi)  = i\rho^2 t  + i\rho (x \cdot \varpi)
\]
and $
b(t,x, \rho, \varpi) = \sum_{j=0}^N \rho^{-j}{b_j(t,x, \varpi)}
$
is smooth amplitude.
In this case, we compute
\begin{align*}
    &\opqstar(e^{-\rho^2 t -\rho(x \cdot \varpi)} b ) \\
    = &e^{e^{-\rho^2 t -\rho(x \cdot \varpi)} b}(
    -\rho^3 \Tzero b
    + \rho^2 \Tonet b
    + \rho \Ttwot b
    + (P+q) b),
\end{align*}
where $\Tzero$ is as before and we define
\[
\Tonet = -\partial_t + \Delta - 1,
\qquad \Ttwot = 2 ( \omega \cdot \nabla  - \omega \cdot \nabla \partial_t).
\]
Then $b_0$ solves the same transport equation (\ref{transport_0}) and therefore
can be given by (\ref{eq_a0}).
Other terms $b_1, \ldots, b_N$ can be constructed in a similar way.

\section{The recovery from  all boundary measurements}\label{sec_allbd}
In this section we would like to use the geometric optics solution constructed above to
solve the inverse problem.
For this purpose, we need to first prove it is indeed an approximate solution to the linear problem, in the sense that its remainder term is relatively small, for sufficiently large $\rho$.
\subsection{The remainder term}
Suppose $v_N$ is constructed as in Section \ref{sec_goc}.
In this part, we construct a remainder term $r_N$ such that
\[
\opq(v_N + e^{i\varphi} r_N) = 0
\]
in $(0, T) \times \Omega$,
with $r_N$ relatively small.
Recall
we write $P = \partial^2_t - \Delta - \partial_t\Delta$. Its symbol is given by
\[
P(\tau, \xi) = -\tau^2 + \xi \cdot \xi + i\tau \xi \cdot \xi,
\qquad \text{for }(\tau, \xi) \in \mathbb{R} \times \mathbb{R}^n.
\]
For a multi-index $\alpha = (\alpha_0, \alpha_1, \ldots, \alpha_n)$,
recall we write
\[
P^{(\alpha)}(\tau, \xi) = \partial^\alpha P(\tau, \xi), \qquad
\tP(\tau, \xi) = (\sum_\alpha |P^{(\alpha)}(\tau, \xi) |^2)^{1/2}.
\]
Then a direct computation shows that
\begin{align}\label{eq_tPgeq}
    \tP^2(\tau, \xi) \geq
    |\partial_\xi P(\tau, \xi)|^2 \geq
    |2 \xi (1 + i\tau)|^2.
\end{align}
Now we prove the following proposition about the remainder term.

\begin{pp}\label{pp_remainder}
    For fixed $N \geq 0$, a large parameter $\rho$, and $\omega \in S^{n-1}$, let $v
    _N(t,x)$ be the approximate solution constructed in (\ref{eq_vN}),
    with the phase function $\varphi = -i \rho^2 t + i\rho(x \cdot \omega)$ and the smooth amplitude  $a =  \sum_{j=0}^N\rho^{-j}{a_j(t,x, \omega)}$,
    where $a_0$ is given by (\ref{eq_a0})
    and $a_j$ satisfies  (\ref{transport_1}, \ref{transport_2},
    \ref{transport_j}) for $j=1, \ldots, N$.
    Then there exists a solution $r_N$ to the equation
    \begin{align}\label{eq_rN}
        (P+q)(e^{i \varphi} r_N) = -(P+q) v_N
        \qquad \text{in }(0, T) \times \Omega
    \end{align}
    such that for $s \geq 0$, we have
    \[
    \|r_N\|_{H^s((0,T)\times \Omega)} \leq {C_N}{\rho^{-N-1}}, 
    \]
    where $C_N$ depends on $s, n, N, q, \Omega$ and the choice of $\phi, \chi$. 
\end{pp}
\begin{proof}
    We are inspired by \cite[Theorem 1.3]{isakov1991completeness}.
    Let $\zetaz = (-i \rho^2,i\rho \omega)$ and we have
    \[
    P^{(\alpha)} (D) e^{i \varphi}= P^{(\alpha)}(\zetaz) e^{i \varphi}. \]
    Then we compute
    \[
    (P+q)(e^{i \varphi } r_N) = e^{i \varphi} (\sum_\alpha
    \frac{1}{\alpha!}P^{(\alpha)}(\zetaz)  D^\alpha r_N
    + q r_N)
    = e^{i \varphi} (P(D + \zetaz) r_N + q r_N).
    \]
    For convenience, we write $\Po(D) = P(D + \zetaz)$.
    Equation (\ref{eq_rN}) can be written as
    \begin{align}\label{eq_wR}
        P_o(D) r_N = -(\rho^{-N+2} F_N  + q r_N)
    \end{align}
    where we write $ (P+q) v_N = \rho^{-N+2} e^{i\varphi} F_N$ in (\ref{eq_errorterm}).

    Now let $E$ be the bounded linear operator from Proposition \ref{pp_Eestimate} for $\Po(D)$, $X = (0, T) \times \Omega$, and $Q(D) = 1$.
    For $r \in H^s(X)$,
    we consider the map $\mathcal{J}(r) = -E(\rho^{-N+2} F_N + q r)$.
    Note that
    by (\ref{eq_Eestimate}), one has
    \[
    \|\mathcal{J}(r)\|_{H^s(X)}
    \leq C\sup_{\xi \in \mathbb{R}^n} \frac{1}{\widetilde{\Po}(\xi)}(\rho^{-N+2}\|F_N\|_{H^s(X)}  + \|q\|_{C^{\lceil s \rceil}([0,T] \times \bar{\Omega})} \|r\|_{H^s(X)}).
    \]
    By (\ref{eq_tPgeq}), for $(\tau, \xi) \in \mathbb{R} \times \mathbb{R}^n$, we have
    \begin{align*}
        \tP_o^2(\zeta) = \tP^2(\zeta + \zetaz)
        &\geq |2(\xi + i\rho \omega)(1 + i(\tau - i \rho^2)|^2 \\
        &= 4|\xi + i\rho \omega|^2 |(1 + \rho^2) +  i\tau |^2\\
        & =4(|\xi|^2 + \rho^2)((1+ \rho^2)^2 + \tau^2) \geq 4 \rho^6.
    \end{align*}
    It follows that
    \begin{align}\label{eq_Jw}
        \|\mathcal{J}(r)\|_{H^s(X)}
        \leq C(\rho^{-N-1}\|F_N\|_{H^s(X)}  + \rho^{-3}\|q\|_{C^{\lceil s \rceil}([0,T] \times \bar{\Omega})} \|r\|_{H^s(X)}).
    \end{align}
    We choose $\rho$ large enough such that $ \rho > \max\{2, C\|q\|_{C^{\lceil s \rceil}([0,T] \times \bar{\Omega})}, CC_{F_N}\}$.
    This implies
    \[
    \|\mathcal{J}(r)\|_{H^s(X)}
    \leq (1+ \|r\|_{H^s(X)}))/2 \leq 1,
    \]
    when $\|r\|_{H^s(X)} \leq 1$.
    Thus, the operator $\mJ$ maps the ball $B_1 = \{r: \|r\|_{H^s(X)} \leq 1\}$ to itself.
    Next, we prove $\mJ$ is contraction when $\rho$ is sufficiently large.
    Indeed, for $r_1, r_2 \in B_1$, we compute
    \begin{align*}
        \|\mJ(r_2) - \mJ(r_1)\|_{H^s(X)}
        = & \|E(q(r_2 - r_1))\|_{H^s(X)} \\
        \leq & \rho^{-3} C \|q\|_{{C^{\lceil s \rceil}([0,T] \times \bar{\Omega})}} \|r_2 - r_1\|_{H^s(X)}
        < \frac{1}{2}\|r_2 - r_1\|_{H^s(X)}.
    \end{align*}
    Then by Banach's contraction theorem, there is a function $r_N$ in $B_1$ satisfying $r_N = \mJ(r_N)$, which solves
    (\ref{eq_wR}).
    Note that $r_N$ satisfies the inequality (\ref{eq_Jw}), and therefore we have
    \[
    \frac{1}{2}\|r_N\|_{H^s(X)} \leq
    \|\mathcal{J}(r_N)\|_{H^s(X)} -  \rho^{-1}\|r_N\|_{H^s(X)}
    \leq  C \rho^{-N-1}\|F_N\|_{H^s(X)}.
    \]

\end{proof}

\begin{remark}
    One can consider other linear phase functions and the corresponding geometric optics construction, for example, the exponentially decaying solution given by  $\varphi = i \rho t + \rho (\omega \cdot x)$.
    We emphasize that we still have the control of $\|r_N\|_{H^s((0, T) \times \Omega)}$, but we do not have the vanishing initial conditions, as the support of $E_0$ is unknown.
\end{remark}

\subsection{The backward problem}\label{subsec_backward_rm}
Consider the geometric optics solution constructed for the backward problem in Section \ref{subsec_backward_goc}.
To construct the remainder term, we write $\Pad = \partial^2_t - \Delta + \partial_t\Delta$. Note it has the symbol
\[
\Pad(\tau, \xi) = -\tau^2 + \xi \cdot \xi - i\tau \xi \cdot \xi,
\qquad \text{for }(\tau, \xi) \in \mathbb{R} \times \mathbb{R}^n.
\]
We have
\begin{align*}
    \tPad^2(\tau, \xi) \geq
    |2 \xi (1 - i\tau)|^2.
\end{align*}
Following the idea of Proposition \ref{pp_remainder}, we prove the following corollary.
\begin{crl}\label{cr_remainder}
    For fixed $N \geq 0$, a large parameter $\rho$, and $\varpi \in \mathbb{S}^{n-1}$, let $w
    _N(t,x)$ be the approximate solution in (\ref{eq_wN}),
    with the phase function $\psi = i \rho^2 t + i\rho(x \cdot \varpi)$ and the smooth amplitude  $b =  \sum_{j=0}^N\rho^{-j}{b_j(t,x, \varpi)}$ constructed above, where $b_0$ is given by (\ref{eq_a0}).
    Then there exists a solution $d_N$ to the equation
    \begin{align}\label{eq_rN_pstar}
        (P^*+q)(e^{i \psi} d_N) = -(P^*+q) w_N \qquad \text{in }(0,T)\times \Omega
    \end{align}
    such that
    \[
    \|d_N\|_{H^s((0,T)\times \Omega)} \leq {C_N}{\rho^{-N-1}}, 
    \]
    where $C_N$ depends on $s, n, N, q, \Omega$ and the choice of $\phi, \chi$. 
\end{crl}
\begin{proof}
    Let $\zetaz = (i \rho^2, i\rho \varpi)$.
    It suffices to verify $\tPad^2(\zeta+ \zetaz) \geq C \rho^6$.
    Indeed, we have
    \begin{align*}
        \tPad^2(\zeta + \zetaz)
        &\geq
        |2(\xi + i\rho \varpi)(1 - i(\tau + i \rho^2)|^2
        = 4|\xi + i\rho \varpi|^2 |(1 - \rho^2) -  i\tau |^2\\
        & =4(|\xi|^2 + \rho^2)((1- \rho^2)^2 + \tau^2) \geq 4 \rho^2(\rho^4 - 2 \rho^2 + 1) \geq 2 \rho^6,
           \end{align*}
       for sufficiently large $\rho$.
\end{proof}
\subsection{Recovering $q$}
We consider the recovery of $q$ from the linear problem, which is related to the first-order linearization of $L_{q, \beta}$.
More explicitly, consider the initial-boundary value problem
\begin{equation}\label{eq_linear}
    \begin{aligned}
        \opq u &= 0, & \  & \mbox{on } \M,\\
        u(t,x) &= h, & \ &\mbox{for  } x\in \partial \Omega,\\
        u = g_0, \  \partial_t u &= g_1, & \ &\mbox{for  } t=0\\
    \end{aligned}
\end{equation}
Let  $(g_0, g_1, h) \in \Gm(\rho, \rho_h)$.
Then this linear problem has a unique solution $u$ satisfying  (\ref{est_nl}),
if the data $(g_0, g_1, h)$ satisfies the $m$th-order compatibility condition (\ref{def_comp_nl}) with $\beta \equiv 0$,
by Proposition \ref{pp_nl}.
We emphasize without the nonlinear term, we do not need to assume the data are sufficiently small.
Thus, for such $(g_0, g_1, h)$, we define the map
\[
\mLlin_{q}: (g_0, g_1, h) \rightarrow (u(T), \partial_t u(T), \partial_\nu u|_{\partial \Omega}).
\]

%
First, we prove the following proposition, which shows all boundary data for the linear problem determines the potential $q$.
\begin{pp}\label{pp_recoverq}
Let $q^{(j)} \in {C^\infty([0,T]\times \bar{\Omega})}$ be two potentials, $j = 1,2$.
For some $\rho, \rho_h>0$, suppose
\[
\mLlin_\qone(g_0,g_1,h) = \mLlin_\qtwo(g_0,g_1,h)
\]
for any $(g_0,g_1,h)\in \Gm(\rho, \rho_h)$ satisfying the $m$th-order compatibility condition (\ref{def_comp_nl}), where $\beta \equiv 0$.
Then we have $\qone = \qtwo$ in $[0,T]\times \bar{\Omega}$.
\end{pp}
\begin{proof}
We are inspired by \cite[Theorem 3.3]{isakov1991completeness}.
For large $\rho$ and fixed $\omega \in S^{n-1}$,
let $\uone = e^{i\varphi}(a_0 + r_0)$ be constructed as in Proposition \ref{pp_remainder},
which solves
\[
(P + \qone) \uone = 0
\]
with the phase function $\varphi = -i\rho^2 t + i\rho(x \cdot \omega)$.
Recall in (\ref{eq_a0}), we choose
\[
a_0(t, x,\omega) = \phi(t) \prod_{j = 2}^{n} \chi(\omega_j \cdot (x - y_0)) \coloneqq \phi(t)\tilde{a}_0(x, \omega).
\]
where $\phi \in  C_0^\infty((0, T))$ and $\chi \in C_0^\infty(\mathbb{R})$ is supported in $(-\ep, \ep)$ with $\chi = 1$ near $0$.
Note that $\tilde{a}_0(x, \omega)$ is
a smooth function supported in a small $\ep$-neighborhood of the ray $\gamma(s) = s \omega + y_0$, for a fixed point $y_0 \in \partial \Omega$.
Recall the remainder term satisfies $\|r_0\|_{H^s((0,T) \times \Omega)} \leq C\rho^{-1}$, where we choose $s$ large enough.
This implies $\uone \in H^{s}((0,T) \times \Omega)$.
Then we set
\[
g_0 = \uone(0),\qquad  g_1 = \partial_t \uone(0),\qquad h = \uone|_{\partial \Omega},
\]
which belongs to {$H^{2m+1}(\Omega)\times H^{2m-1}(\Omega) \times C^{2m+2}(\partial \Omega)$} and satisfies the compatibility condition (\ref{def_comp_nl}) with $\beta \equiv 0$.
Thus, there exists a unique
$\utwo$ solving the linear problem (\ref{eq_linear}) with the potential $\qtwo$ and the initial-boundary data $(g_0, g_1, h)$.
We set $u = \utwo - \uone$ and $q = \qtwo - \qone$, which satisfies
\begin{equation*}
    \begin{aligned}
        (P+\qtwo) u &= q\uone, & \  & \mbox{on } \M,\\
        u(t,x) &= 0, & \ &\mbox{for  } x\in \partial \Omega,\\
        u = 0, \  \partial_t u &= 0, & \ &\mbox{for  } t=0.\\
    \end{aligned}
\end{equation*}
Moreover, with the assumption that
$\mLlin_\qone(g_0, g_1, h) = \mLlin_\qtwo(g_0, g_1, h)$, we have
\[
\partial_\nu u|_\Omega = 0, \qquad u(T,x) = \dt u(T,x) = 0.
\]
Now let $w = e^{i\psi}(b_0 + d_0)$ be as in Corollary \ref{cr_remainder},
which solves the backward problem
\[
(\Pad + \qtwo) w = 0
\]
with the phase function $\psi = i\rho^2 t + i\rho(x \cdot \varpi)$.
Here we set $\varpi = -\omega$.
We choose
\[
b_0(t, x,\varphi) = \phi(t)\prod_{j = 2}^{n} \chi(\omega_j \cdot (x - y))
\coloneqq \phi(t)\tilde{b}_0(x, \omega).
\]
as a smooth function supported in a small $\ep$-neighborhood of the same ray $\gamma(s)$, with the same $\phi$.
In this case, the remainder term satisfies
$\|d_0\|_{H^s((0,T) \times \Omega)} \leq C \rho^{-1}$, for sufficiently large $s$.  
We multiply $w$ with the equation above and integrate over $t,x$ to have
\begin{align*}
0 = \int_0^T \int_\Omega (P+\qtwo) uw - q\uone w \diff x \diff t
= \int_0^T \int_\Omega u (\Pad+\qtwo) w - q\uone w \diff x \diff t,
\end{align*}
which implies
\[
\int_0^T \int_\Omega - q\uone w \diff x \diff t = 0.
\]
Then we plug in the expansion of $\uone, w$ to have
\begin{align*}
\mI_q = \int_0^T \int_\Omega q a_0 b_0 \diff x \diff t
=  -\int_0^T \int_\Omega q ( r_0 b_0 + a_0 d_0)  \diff x \diff t
\leq C \rho^{-1}.
\end{align*}
When $\rho \rightarrow \infty$, this implies $\mathcal{I}_q = 0$.
Since $\phi \in C_0^\infty((0,T))$ is arbitrary, we must have
\[
 \int_\Omega  q(t,x)
 \tilde{a}_0 \tilde{b}_0\diff x = 0, \qquad t \in [0,T].
\]
Note that $\tilde{a}_0$ and $\tilde{b}_0$ are both supported in a small $\ep$-neighborhood of the ray $\gamma(s) = s \omega + y$.
Let $\ep \rightarrow 0$ and we can extract the line integral
\[
X q(\gamma) = \int q(t, \gamma(s)) \diff s = 0, \qquad t \in [0,T].
\]
Since $\omega \in S^{n-1}$ and $y_0\in \partial \Omega$ are arbitrary,
the X-ray transform of $q$ over all rays vanishes.
With the assumption on $\Omega$, we have $q = \qtwo - \qone = 0$.
Indeed, the X-ray transform is injective on $L^1(\Rn)$ by the Fourier Slice Theorem, for example, see \cite[Chapter 2]{book_SU}.
One can extend $q \in C^\infty([0,T] \times \bar{\Omega})$ to $L^1(\Rn)$ by setting it equal to zero in $\Rn \setminus \bar{\Omega}$, as $\Omega$ is bounded.
\end{proof}

Moreover, when we have $\dt^j \beta(0,x) = 0$ for $x \in \Omega$ and $j = 0, \ldots, m$, the data $(g_0, g_1, h)$ with the $m$th-order compatibility for the linear problem also satisfies the $m$th-order compatibility for the nonlinear problem, see (\ref{def_comp_nl}).
In this case, we can say
\begin{align}\label{eq_Lin}
    \mLlin_{q}(g_0, g_1, h) = 
    \partial_\vep \mL_{q, \beta}(\vep g_0, \vep g_1, \vep h)|_{\vep=0},
\end{align}
for small $(g_0, g_1, h) \in \Gm$.
Therefore, Proposition \ref{pp_recoverq} implies the first-order linearization of $L_{q,\beta}$ determines $q$, under this assumption on $\beta$.
\subsection{The second-order Linearization}\label{subsec_2nd}
In this subsection, let $\vep_1, \vep_2 > 0$ be small parameters and let
$u_{\vep_1,\vep_2}$ be the solution to the nonlinear problem (\ref{eq_nl}) with data
\begin{align}\label{def_2data}
    (g_0, g_1, h) = \vep_1 (g_{0,1}, g_{1,1}, h_1)  + \vep_2 (g_{0,2}, g_{1,2}, h_2).
\end{align}
We consider the second-order linearization
\[
U_2 = \partial_{\vep_1} \partial_{\vep_2} u_{\vep_1, \vep_2}|_{\vep_1 = \vep_2 = 0}.
\]
Let $u_j$ be solutions to the linear problem
\begin{align}\label{LS}
    \begin{aligned}
        (P+q) u_j(t,x)  &= 0, & \  & \mbox{on } \M,\\
        u_j(t,x) &= f_j, & \ &\mbox{for  } x\in \partial \Omega,\\
        u_j = g_{0,j},\  \dt u_j &= g_{1,j}, & \ &\mbox{for  } t=0.\\
    \end{aligned}
\end{align}
Then $U_2$ solves
\begin{align*}
    \begin{aligned}
        (P+q) U_2  &= \beta(t,x) \dt^2 (u_1u_2), & \  & \mbox{on } \M,\\
        U_2(t,x) &= 0, & \ &\mbox{for  } x\in \partial \Omega,\\
        U_2 = \partial_t U_2 &= 0, & \ &\mbox{for  } t=0.\\
    \end{aligned}
\end{align*}
Now for any $u_0$ solving $(\Pad + q) u_0 = 0$,
we integrate by parts to have
\begin{align}\label{eq_Ibeta}
    \int_0^T \int_\Omega \beta(t,x) \partial_t^2(u_1 u_2) u_0 \diff x \diff t
    =  
    \int_\Omega \dt U_2(T) &u_0(T) -  U_2(T)(\dt u_0(T) - \Delta u_0(T)) \diff x\\
    &- \int_0^T (\partial_\nu U_2 u_0)|_{\partial \Omega}
    + (\partial_\nu \dt U_2 u_0)|_{\partial \Omega} \diff t. \nonumber
\end{align}
\subsection{Recovering $\beta$.}\label{subsec_beta}
Now we would like to recover the nonlinear coefficient from all boundary measurements.
Suppose we have two coefficients $\betaone, \betatwo \in C^{\infty}([0,T] \times \bar{\Omega})$.
Let $\uk_{\vep_1,\vep_2}$ be the solutions to the nonlinear problem (\ref{eq_nl})  with $\betak$ and data $(g_0, g_1, h)$ given in (\ref{def_2data}), for $k=1,2$.
We consider the corresponding source-to-solution maps $\mL_\qbk$ and suppose
\[
\mL_\qbone(g_0, g_1, h) = \mL_\qbtwo(g_0, g_1, h),
\]
for any small and compatible data $(g_0, g_1, h)$.

First, according to (\ref{eq_Lin}) and Proposition \ref{pp_recoverq}, we have $\qone = \qtwo$.
Next, we would like to construct approximate solutions $u_j$ to the linear problems using Section \ref{sec_goc}, which are corresponding to the boundary data $(g_{0,j}, g_{1,j}, f_j)$, for $j = 1,2,0$.
For this purpose,
let $\omega_1, \omega_2, \ldots, \omega_{n}$ be an orthonormal basis for $\mathbb{R}^n$.
Let $\varpi =  - (\omega_1 + \omega_2)/\sqrt{2} \in S^{n-1}$ and we choose $\varpi_2, \ldots \varpi_n$ such that they form an orthonormal basis with $\varpi$.
For a fixed point $x_0 \in \Omega$,
we choose $y_1, y_2, y_0\in \partial \Omega$ such that the three rays
\[
\gamma_j(s) = s \omega_j + y_j, 
\ \text{ for } j = 1,2,
\qquad \text{and }
\gamma_0(s) = s \varpi + y_0,
\]
intersect at $x_0$.
We construct approximate solutions
\begin{align*}
    u_{j}  &=  e^{\rho^2t - \rho (w_j \cdot x)} (a_{0,j} + r_{0,j}) ,\quad j = 1,2,
\end{align*}
to $(P + \qone) u_j =(P + \qtwo) u_j  = 0$,
according to Proposition \ref{pp_remainder}.
Here the amplitudes given by
\begin{align*}
    a_{0,1}(t,x,\omega_j) =  \phi(t) \prod_{i \neq 1} \chi(\omega_i \cdot (x - y_1)) \coloneqq \phi(t)\tilde{a}_{0,1}(x, \omega),\\
    a_{0,2}(t,x,\omega_j) =  \phi(t) \prod_{i \neq 2} \chi(\omega_i \cdot (x - y_2)) \coloneqq \phi(t)\tilde{a}_{0,2}(x, \omega),
\end{align*}
are smooth functions supported in a small $\ep$-neighborhood of the ray $\gamma_1(s)$ and $\gamma_2(s)$ respectively.
Moreover, the remainder term is given by  $\|r_{0,j}\|_{H^s((0,T) \times \Omega)} \leq C\rho^{-1}$, for $j = 1,2$, where $s$ is chosen to be large enough.
Then we set
\[
(g_{0,j}, g_{1,j}, h_j) = (u_j(0), \dt u_j(0), u_j|_{\partial \Omega}), \quad j =1,2,
\]
which satisfy the compatibility condition (\ref{def_comp_nl}), under the assumption $\dt^j \beta(0,x) = 0$ for $x \in \Omega$ and $j = 0, \ldots, m$.
For large $s$, we have $(g_0, g_1, h) \in \Gm$.
Then with small $\vep_1, \vep_2$, the initial-boundary data $(g_0, g_1, h)$ given by (\ref{def_2data}) allows a unique solution to the nonlinear problem.
Thus, let $\uk_{\vep_1,\vep_2}$ be the solutions to the nonlinear problem (\ref{eq_nl})  with $\betak$ and   (\ref{def_2data}).
Note the second-order linearization $U_2^{(k)}$ satisfies
\[
U_2^{(1)} (T) = U_2^{(2)} (T), \quad
\dt U_2^{(1)}(T) = \dt U_2^{(2)}(T), \quad
\partial_\nu U_2^{(1)}|_{\partial \Omega} = \partial_\nu U_2^{(2)}|_{\partial \Omega},
\]
since $\mL_\qbone(g_0, g_1, h) = \mL_\qbtwo(g_0, g_1, h)$.

Next, we construct approximate solutions $u_0$ to the backward problem $(\Pad + \qone) u_0 =(\Pad + \qtwo) u_0  = 0$ by
\begin{align*}
    u_{0}  &=  e^{-2\rho^2t - \sqrt{2}\rho (\varpi \cdot x)} (b_{0} +  d_{0}),
\end{align*}
according to Corollary \ref{cr_remainder}.
Here the amplitudes given by
\begin{align*}
    b_{0}(t,x,\varpi) =  \phi(t) \prod_{i \neq 1} \chi(\varpi_i \cdot (x - y_0)) \coloneqq \phi(t)\tilde{b}_{0}(x, \varpi),
\end{align*}
is a smooth function supported in a small $\ep$-neighborhood of the ray $\gamma_0(s)$.
In this case, the remainder term is given by  $\|d_{0}\|_{H^s((0,T) \times \Omega)} \leq C\rho^{-1}$, where $s$ is large enough.

Now we plug in $u_1,u_2,u_0, U_2^{(1)}, U_2^{(2)}$ in (\ref{eq_Ibeta}) to have
\begin{align*}
    \int_0^T \int_\Omega (\betaone- \betatwo)  \ddt(u_1 u_2) u_0 \diff x \diff t = 0.
\end{align*}
We compute
\begin{align*}
    \ddt(u_1 u_2) &=  \ddt (e^{2 \rho^2 t - \rho (\omega_1 + \omega_2) \cdot x} (a_{0,1} +r_{0,1})(a_{0,2} +  r_{0,2}))
    = e^{2 \rho^2 t -\rho (\omega_1 + \omega_2) \cdot x} (4 \rho^4  \phi^2(t) \tilde{a}_{0,1}\tilde{a}_{0,2} + R_1),
\end{align*}
where we write
\begin{align*}
    R_1 &= 4 \rho^4 (r_{0,1}(a_{0,2} +  r_{0,2}) + (a_{0,1} + r_{0,1}) r_{0,2})\\
    &-2 \rho^2 \dt((a_{0,1} + r_{0,1})(a_{0,2}  + r_{0,2})) + \ddt ((a_{0,1} +  r_{0,1})(a_{0,2} + r_{0,2})).
\end{align*}
A straightforward computation shows that $\|R_1\|_{H^{s-2}((0,T)\times \Omega)} \leq C \rho^3$.
Then we have
\begin{align*}
    0 &= \int_0^T \int_\Omega (\betaone- \betatwo)    (4 \rho^4  \phi^2(t) \tilde{a}_{0,1}\tilde{a}_{0,2} + R_1) (b_{0} + d_{0}) \diff x \diff t \\
    & = \int_0^T \int_\Omega (\betaone- \betatwo)    (4 \rho^4  \phi^3(t) \tilde{a}_{0,1}\tilde{a}_{0,2} \tilde{b}_{0}+ R_2) \diff x \diff t
\end{align*}
where we write
\[
R_2  = R_1 (b_{0} + d_{0})  + (4 \rho^4  \phi^2(t) \tilde{a}_{0,1}\tilde{a}_{0,2} + R_1) d_{0}.
\]
Similarly, we have $\|R_2\|_{H^{s-2}((0,T)\times \Omega)}  \leq C \rho^3$.
It follows that
\begin{align*}
    \int_0^T \int_\Omega (\betaone- \betatwo)   \phi^3(t) \tilde{a}_{0,1}\tilde{a}_{0,2} \tilde{b}_{0} \diff x \diff t = 0.
\end{align*}
Since $\phi \in C_0^\infty((0,T))$ is arbitrary, this implies
\[
\int_\Omega (\betaone (t,x)- \betatwo(t,x)) \tilde{a}_{0,1}\tilde{a}_{0,2} \tilde{b}_{0}
\diff x = 0, \qquad t \in [0,T].
\]
Recall the rays $\gamma_1, \gamma_2, \gamma_0$ intersect at the fixed point $x_0 \in \Omega$.
Then we can assume $\tilde{a}_{0,1}\tilde{a}_{0,2} \tilde{b}_{0}$ is supported in a small neighborhood $N_\ep$ of $x_0$, which is contained in a ball of radius $\ep$.
By shrinking the support of $\chi$,
we have
\[
\betaone(t,x_0)- \betatwo(t,x_0) = \lim_{\ep \rightarrow 0} \frac{1}{|N_\ep|}\int_\Omega (\betaone(t,x)- \betatwo(t,x)) \tilde{a}_{0,1}\tilde{a}_{0,2} \tilde{b}_{0}
\diff x
= 0.
\]
Since $x_0 \in \Omega$ is arbitrary, we must have $\betaone = \betatwo$ in $[0,T] \times \bar{\Omega}$.

\section{Recovery from the DN map}\label{sec_gocnew}
In this section, we would like to improve the results in Section \ref{sec_allbd} using fundamental solutions constructed for the characteristic Cauchy problem, see Section \ref{subsec_cauchy}.
We prove the recovery of $q, \beta$ from the lateral boundary  measurements.
\subsection{The reminder term}\label{subsec_reminder_new}
We reconsider the approximate solution constructed for the linear problem.
The aim is to construct a reminder term $r_N$ satisfying
\[
\opq(v + e^{i\varphi} r_N) = 0 \qquad \text{in }(0, T) \times \Omega
\]
with $r_N$ relatively small,
as well as, with the zero initial conditions.
Recall
we write $P = \partial^2_t - \Delta - \partial_t\Delta$.
Now we prove the following proposition about the reminder term.
\begin{pp}\label{pp_reminder_new}
    For $N \geq 0$, a large parameter $\rho$, and $\omega \in \mathbb{S}^{n-1}$, let $v
    _N(t,x)$ be the approximate solution constructed in (\ref{eq_vN}),
    with the phase function $\varphi = -i \rho^2 t + i\rho(x \cdot \omega)$ and the smooth amplitude  $a =  \sum_{j=0}^N\rho^{-j}{a_j(t,x, \omega)}$,
    where $a_0$ is given by (\ref{eq_a0})
    and satisfies  (\ref{transport_1}, \ref{transport_2},
    \ref{transport_j}) for $j = 1, \ldots, N$.
    Then there exists a solution $r_N$ to the equation
    \begin{align}\label{eq_rN_new}
        (P+q)(e^{i \varphi} r_N) = -(P+q) v_N
        \qquad \text{in }(0, T) \times \Omega
    \end{align}
    {with $r_N(0,x) = \partial_t r_N(0, x) = 0$},
    such that for $s \geq 0$, we have
    \[
    \|r_N\|_{H^s((0,T)\times \Omega)} \leq {C_N}{\rho^{-N-1}}, 
    \]
    where $C_N$ depends on $q, n, \Omega$ and the choice of $\phi, \chi$.
\end{pp}
\begin{proof}
    We would like to show the reminder term $r_N$ is relatively small for large $\rho$ and additionally it is supported in $t \geq \vep$ for some small number $\vep > 0$. This implies it satisfies vanishing initial conditions at $t = 0$.
    For this purpose, we recall some conclusions in Section \ref{subsec_cauchy} and modify the notations there for our model.
    In the following, let
    \begin{align*}
        N = (1, 0, \ldots, 0) \in \mathbb{R} \times \Rn, \qquad H = \{(t,x): (t,x) \cdot N \geq \vep\}.
    \end{align*}
    Note that here we translate the half space by a parameter $\vep$ independent of $\rho$ and it does not affect the estimate in Proposition \ref{pp_Eestimate_new}.
    As before, we rewrite (\ref{eq_rN_new}) using $\zetaz = (-i\rho, i\rho \omega)$ to have
    \[
    (P+q)(e^{i \varphi } r_N)
    = e^{i \varphi} (P(D + \zetaz) r_N + q r_N)
    = e^{i \varphi} (\Po r_N + q r_N).
    \]
    Equation (\ref{eq_rN_new}) can be written as
    \begin{align}\label{eq_wR_new}
        P_o(D) r_N = -(\rho^{-N+2} F_N  + q r_N)
    \end{align}
    where we write $ (P+q) v_N = \rho^{-N+2} e^{i\varphi} F_N$ in (\ref{eq_errorterm}).
    Note $ F_N$ is supported in $(0, T) \time \Omega$ with respect to $t$.
    Then there exists  $\vep > 0$ such that its support in $t$ is contained in the set of $(\vep, T)$.
    Thus, we have $\supp F_N \subset H$.

    Note for $P_o(D)$, the boundary $\partial H$ is characteristic, since its principal symbol satisfies $P_{o,3}(N) = 0$.
    There is no uniqueness for the initial value problem $P_o(D) u = f$ with $f$ supported in $H$.
    As  mentioned before, we would like to find a solution supported in $H$ with the desired estimate, using Proposition \ref{pp_Eestimate_new}.
    For this purpose, we prove in Lemma \ref{lm_A1A2} that for our operator $P_o$, there exist $A_1, A_2$ independent of $\rho$ such that condition (\refb) is satisfied.
    With this lemma, we set $A_1 = 1$ and $A_2 = 0$. Thus,
    we have $P_o(D)$ satisfies (\refb) with $A_1 = A_2 + 1$, which is independent of $\rho$.

    Now let $X = (0, T) \times \Omega$ and $\tX$ be a small open neighborhood of $X$.
    Let $X_0$ be a small open neighborhood of $\tX- \tX$, which is bounded in $\Rn$.
    Then we choose $E_H$ as the bounded linear operator in Proposition \ref{pp_Eestimate_new} for $\Po(D)$ with $X, \tX, X_0$.
    For every $r \in H^s(\tX)$ with $\supp r \subset H$,
    we consider the map $\mJ(r) = -E_H(\rho^{-N+2}F_N + q r)$,
    which has support in $H$ and satisfies
    \[
    P_o(D)\mJ(r) = - (\rho^{-N+2}F_N + q r) \quad \text{in } X,
    \]
    with the estimate
    \[
    \| \mJ(r)\|_{H^s(\tX)} \leq C  \sup_{\zeta \in \mathbb{R}^{n+1}} \frac{1}{\tP_o(\zeta)} \|\rho^{-N+2}F_N + q r\|_{H^s(\tX)},
    \]
    where $C$ is independent of $\rho$.
    On the other hand, we compute
    \begin{align*}
        \tP_o^2(\zeta) = \tP^2(\zeta + \zetaz)
        \geq |2(1 + i(\tau - i\rho^2))(\xi + i\rho \eta)|^2
        = 4((1 + \rho^2)^2 + \tau^2)(|\xi|^2 + \rho^2)  \geq 4 \rho^6.
    \end{align*}
    It follows that for any $r \in H^s(\tX)$ supported in $H$, we have
    \begin{align}\label{eq_Jw_new}
        \| \mJ(r)\|_{H^s(\tX)} \leq  {C}({\rho}^{-N-1} \|F_N\|_{H^s(\tX)}  + \rho^{-3}\|q\|_{C^{\lceil s \rceil}([0,T] \times \bar{\Omega})} \|r\|_{H^s(\tX)}).
    \end{align}
    where we extend $q\in C^{\infty}(\bar{X})$ to a function in  $C^{\infty}(\tX)$.
    This motivates us to define
    \[
    H^{s,H}(\tX) = \{u: u \in H^s(\tX) \text{ with $\supp u \subset H$}\}.
    \]
    If we choose  $\rho$ large enough such that $ \rho > \max\{2C\|q\|_{C^{\lceil s \rceil}([0,T] \times \bar{\Omega})} , 2CC_{F_0}\}$, then we have
    \[
    \|\mJ (r)\|_{H^s(\tX)} \leq (1+  \|r\|_{H^s(\tX)})/2 \leq 1,
    \]
    for any $\|r\|_{H^s(\tX)} \leq 1$.
    Thus, the operator $\mJ$ maps the ball $B_1 = \{r \in H^{s,H}(\tX): \|r\|_{H^{s}(\tX)} \leq 1\}$ to itself.
    Next, we prove $\mJ$ is a contraction for large $\rho$.
    Indeed, for $r_1, r_2 \in H^s(\tX)$,
    we compute
    \begin{align*}
        \|\mJ(r_2) - \mJ(r_1)\|_{H^s(\tX)}
        = & \|E_H(q(r_2 - r_1))\|_{H^s(\tX)} \\
        \leq & \rho^{-3} C \|q\|_{C^{\lceil s \rceil}([0,T] \times \bar{\Omega})}  \|r_2 - r_1\|_{H^s(\tX)}
        < \frac{1}{2}\|r_2 - r_1\|_{H^s(\tX)}.
    \end{align*}
    Then by Banach's contraction theorem, there is a function $r_N$ in $B_1$ satisfying $r_N = \mJ(r_N)$, which solves
    (\ref{eq_rN_new}).
    Note that $r_N$ satisfies the inequality (\ref{eq_Jw_new}), and therefore we have
    \[
    \frac{1}{2}\|r_N\|_{H^s(\tX)} =
    \|\mathcal{J}(r_N)\|_{H^s(\tX)} -  \frac{1}{2}\|r_N\|_{H^s(\tX)}
    \leq  C \rho^{-N-1}\|F_N\|_{H^s(\tX)}.
    \]
    In particular, this implies $\|r_N\|_{H^s(X)} \leq  C_N \rho^{-N-1}\|F_N\|_{H^s(\tX)},$ for some constant $C_N$ independent of $\rho$.
\end{proof}

\begin{lm}\label{lm_A1A2}
    Let $P$, $\zeta_0$, $N$ be defined as above.
    Suppose $\sigma(\zeta)$ is a solution to $\Po(\zeta + \sigma N)$, where $\zeta = (\tau, \xi) \in \mathbb{C} \times \mathbb{C}^n$.
    If $\sigma(\zeta)$ is analytic and single-valued in a ball $B \subset \mathbb{C} \times \mathbb{C}^n$ with real center and radius $A_1 = 1$,
    then we have
    \[
    \sup_{\zeta \in B} \mathrm{Im} (\sigma(\zeta)) \geq 0.
    \]
\end{lm}
\begin{proof}
    First, we would like to solve $\tau(\xi)$ from $P(\tau, \xi) = 0$ for $\xi \in \mathbb{C}^n$.
    Indeed, the equation $-\tau^2 + i \tau \xi \cdot \xi + \xi \cdot \xi = 0$ have two complex roots
    \begin{align*}
        \tau_\pm(\xi)
        = \frac{1}{2}(iR(\xi)  \pm \sqrt{-R^2(\xi) + 4 R (\xi)}),
        \qquad \text{where }R(\xi) = \xi \cdot \xi.
    \end{align*}
    Recall $\zetaz = (-i \rho^2, i \rho \omega)$.
    Then the roots for $P_o(\zeta + \sigma N)  = P(\zeta + \zeta^0 + \sigma N) = 0$ are given by
    \begin{align*}
        \sigma_\pm(\tau, \xi) &= \frac{1}{2}(iR(\xi + i \rho \omega)
        \pm \sqrt{\triangle(\xi + i \rho \omega)})  + i\rho^2 - \tau.
    \end{align*}
    For a ball $B$ with real center and radius $1$, we can choose a point $(\tau, \xi)\in \mathbb{R} \times \Rn$ in $B$ such that $(\tau, \xi) = (\tau, r \theta)$, where $r \geq 1$ and $\theta \in \mathbb{S}^{n-1}$.
    First we compute
    \[
    R(\xi + i \rho \omega) = (r^2 - \rho^2) + i 2 r \rho (\omega \cdot \theta)
    \coloneqq c + id,
    \]
    and
    \[
    -R^2(\xi) + 4 R (\xi) = (d^2 - c^2 + 4c) + i 2d(2-c) \coloneqq C + iD,
    \]
    where $c, d, C, D \in \mathbb{R}$.
    Suppose $C + iD = (a + ib)^2$ for some real numbers $a, b$.
    Then a straightforward computation shows that
    \[
    b^2 = \frac{1}{2}(-C + \sqrt{C^2 + D^2}).
    \]
    Note that
    \begin{align*}
        C^2 + D^2 &= (d^2 - c^2 + 4c)^2 + 4d^2(2-c)^2 \\
        & = ((d^2 +4) - (c-2)^2)^2 + 4(d^2+4)(c-2)^2 -16(c-2)^2 \\
        & \leq ((d^2 +4) + (c-2)^2)^2.
    \end{align*}
    This implies that
    \[
    b^2 \leq \frac{1}{2}(-(d^2 - c^2 + 4c) +(d^2 +4) + (c-2)^2) = (c-2)^2.
    \]
    Then we have
    \[
    (c + 2\rho^2)^2 - b^2 \geq (r^2 + \rho^2)^2 - (r^2 - \rho^2 -2)^2  = 4(\rho^2 + 1)(r^2 -1) \geq 0.
    \]
    Note $c +2 \rho^2 = r^2 + \rho^2 \geq 0$.
    Thus, the imaginary part of the roots satisfies
    \[
    \mathrm{Im}(\sigma_\pm(0, r\theta)) = \frac{1}{2}(c + 2\rho^2 \pm b) \geq 0.
    \]
    This proves for any ball $B$ with real center and radius $A_1$, we have
    \[
    \sup_{\zeta \in B} \mathrm{Im} (\sigma(\zeta)) \geq A_2 = 0,
    \]
    if $\sigma(\zeta)$ is a real-valued and analytic root.

\end{proof}
\subsection{The backward problem}\label{subsec_backward_new}
Next, we consider the backward problem
\begin{equation}\label{eq_linearback}
    \begin{aligned}
        \opqstar w  &= 0, & \  & \mbox{on } \M,\\
        w(t,x) &= f, & \ &\mbox{for  } x\in \partial \Omega,\\
        w = 0, \  \partial_t w &= 0, & \ &\mbox{for  } t=T.\\
    \end{aligned}
\end{equation}
As in Section \ref{subsec_backward_goc} , for a large parameter $\rho$ and fixed $\varpi \in \mathbb{S}^{n-1}$,
we construct its approximate solution
\begin{align*}
    w(t,x) = e^{i\psi(t, x,\rho, \varpi)} b(t,x,\rho, \varpi).
\end{align*}
For the reminder term,
we follow the same idea of Proposition \ref{pp_reminder_new}.
More explicitly, here we consider
\begin{align*}
    \widetilde{N} = (-1, 0, \ldots, 0) \in \mathbb{R} \times \Rn,
    \qquad \widetilde{H} = \{(t,x): (t,x) \cdot N \geq -T + \vep\},
    \qquad \widetilde{\zetaz} = (i\rho^2, i \rho \varpi).
\end{align*}
Recall we write $\Pad = \partial^2_t - \Delta + \partial_t\Delta$
and $\Pad_o = \Pad(D + \widetilde{\zetaz})$.
We would like to
verify the condition (\refb) for $\Pad_o$.
We have following lemma, as an analog to Lemma \ref{lm_A1A2}.
\begin{lm}\label{lm_A1A2_ad}
    Let $\Pad$, $\widetilde{\zetaz}$, $\widetilde{N}$ be defined as above.
    Let ${\sigma}(\zeta)$ be the solution to $\Pad_o(\zeta + \sigma \widetilde{N})$, where $\zeta = (\tau, \xi) \in \mathbb{C} \times \mathbb{C}^n$.
    If ${\sigma}(\zeta)$ is analytic and single-valued in a ball $B \subset \mathbb{C} \times \mathbb{C}^n$ with real center and radius $A_1 = 1$,
    then we have
    \[
    \sup_{\zeta \in B} \mathrm{Im} (\sigma(\zeta)) \geq 0.
    \]
\end{lm}
\begin{proof}
    In this case, we solve ${\tau}(\xi)$ from $\Pad(\tau, \xi) = 0$ for $\xi \in \mathbb{C}^n$.
    Indeed, the equation $-{\tau}^2 - i {\tau} \xi \cdot \xi + \xi \cdot \xi = 0$ have two complex roots
    \begin{align*}
        \tau_\pm(\xi)
        = \frac{1}{2}(- iR(\xi)  \pm \sqrt{-R^2(\xi) + 4 R (\xi)}),
        \qquad \text{where }R(\xi) = \xi \cdot \xi.
    \end{align*}
    Then the roots for $\Pad_o(\zeta + \sigma N)  = \Pad(\zeta + \widetilde{\zetaz} + \sigma \widetilde{N}) = 0$ are given by
    \begin{align*}
        \sigma_\pm(\tau, \xi)
        &= \tau + i\rho^2-
        {\tau}_\pm(\xi + i\rho \varpi)\\
        &= \frac{1}{2}(iR(\xi + i \rho \varpi)
        \mp \sqrt{-R^2(\xi + i \rho \varpi) + 4 R(\xi + i \rho \varpi)})  + i\rho^2 + \tau.
    \end{align*}
    For a ball $B$ with real center and radius $1$, we can choose a point $(\tau, \xi)\in \mathbb{R} \times \Rn$ in $B$ such that $(\tau, \xi) = (\tau, r \theta)$, where $r \geq 1$ and $\theta \in \mathbb{S}^{n-1}$.
    The proof of Lemma \ref{lm_A1A2} shows that the imaginary part of the roots satisfies
    \[
    \mathrm{Im}({\sigma}_\pm(0, r\theta)) \geq 0.
    \]
    This proves for any ball $B$ with real center and radius $A_1$, we have
    \[
    \sup_{\zeta \in B} \mathrm{Im} ({\sigma}(\zeta)) \geq A_2 = 0,
    \]
    if ${\sigma}(\zeta)$ is a real-valued and analytic root.
\end{proof}

\begin{crl}\label{cr_reminder_new}
    For $N \geq 0$, a large parameter $\rho$, and $\varpi \in \mathbb{S}^{n-1}$, let $w(t,x)$ be the approximate solution constructed in (\ref{eq_vN}),
    with the phase function $\psi = i\rho^2 t + i\rho(x \cdot \varpi)$ and the smooth amplitude $b =  \sum_{j=0}^N\rho^{-j}{b_j(t,x, \varpi)}$ constructed in Section \ref{subsec_backward_goc}.
    Then there exists a solution $d_N$ to the equation
    \begin{align*}
        (\Pad+q)(e^{i \varphi} d_N) = -(\Pad+q) w\qquad \text{in }(0, T) \times \Omega
    \end{align*}
    {with $d_N(T,x) = \partial_t d_N(T, x) = 0$},
    such that for $ s \geq 0$, we have
    \[
    \|d_N\|_{H^s((0,T)\times \Omega)} \leq {C_N}{\rho^{-N-1}},
    \]
    where $C_N$ depends on $q, n, \Omega$ and the choice of $\phi, \chi$.
\end{crl}
\subsection{Recovering $q$}
In this part, we consider the recovery of $q$ from the DN map for the linear problem, which is related to the first-order linearization of $\Lambda_{q, \beta}$.
More explicitly, consider the boundary value problem
\begin{equation}\label{eq_linear_f_new}
    \begin{aligned}
        \opq u &= 0, & \  & \mbox{on } \M,\\
        u(t,x) &= h, & \ &\mbox{for  } x\in \partial \Omega,\\
        u = 0, \  \partial_t u &= 0, & \ &\mbox{for  } t=0.\\
    \end{aligned}
\end{equation}
Let  $h \in C^{2m+2}([0,T] \times \partial \Omega)$.
Then this linear problem has a unique solution $u$ satisfying  (\ref{est_nl}),
if the data $(0, 0, h)$ satisfies the $m$th-order compatibility condition (\ref{def_comp_nl}) with $\beta \equiv 0$,
by Proposition \ref{pp_nl}.
This is equivalent to require $\dt^j h(0) = 0$ for $j = 0, \ldots, m$.
Thus, for $h \in \Hm$, we define the linear DN map
\[
\Lamlin_{q}: h \rightarrow \partial_\nu u|_{\partial \Omega}.
\]


First, we prove the following proposition, which shows the DN map for the linear problem determines the potential $q$.
\begin{pp}\label{pp_recoverq_new}
    Let $q^{(j)} \in {C^\infty([0,T]\times \bar{\Omega})}$ be two potentials, $j = 1,2$.
    For some $\rho, \rho_h>0$, suppose
    \[
    \Lamlin_\qone(h) = \Lamlin_\qtwo(h)
    \]
    for $h \in \Hm$.
    Then we have $\qone = \qtwo$ in $[0,T]\times \bar{\Omega}$.
\end{pp}
\begin{proof}
    We follow exactly the same idea as before.
    For large $\rho$ and fixed $\omega \in S^{n-1}$,
    let $\uone = e^{i\varphi}(a_0 + r_0)$ be constructed as in Proposition \ref{pp_reminder_new},
    which solves
    \[
    (P + \qone) \uone = 0, \qquad u^{(1)}(0) = \dt u^{(1)}(0) = 0,
    \]
    with the phase function $\varphi = -i\rho^2 t + i\rho(x \cdot \omega)$.
    Recall in (\ref{eq_a0}), we choose
    \[
    a_0(t, x,\omega) = \phi(t) \prod_{j = 2}^{n} \chi(\omega_j \cdot (x - y_0)) \coloneqq \phi(t)\tilde{a}_0(x, \omega).
    \]
    where $\phi \in  C_0^\infty((0, T))$ and $\chi \in C_0^\infty(\mathbb{R})$ is supported in $(-\ep, \ep)$ with $\chi = 1$ near $0$.
    Note that $\tilde{a}_0(x, \omega)$ is
    a smooth function supported in a small $\ep$-neighborhood of the ray $\gamma(s) = s \omega + y_0$, for a fixed point $y_0 \in \partial \Omega$.
    Recall the reminder term satisfies $\|r_0\|_{H^s((0,T) \times \Omega)} \leq C\rho^{-1}$, where we choose $s$ large enough. 
    This implies $\uone \in H^{s}((0,T) \times \Omega)$.
    Then we set
    \[
    h = \uone|_{\partial \Omega},
    \]
    which belongs to {$C^{2m+{2}}([0,T] \times \partial \Omega)$}, for large enough $s$.
    In particular, we have $\dt^j h(0) = 0$ for $j = 0, \ldots, m$, as in Proposition \ref{pp_reminder_new} we construct $a_0, r_0$ supported in $(\vep, T)$ for some $\vep > 0$.
    Thus, there exists a unique
    $\utwo$ solving the linear problem (\ref{eq_linear}) with the potential $\qtwo$, the Dirichlet data $h$, and vanishing initial conditions.
    We set $u = \utwo - \uone$ and $q = \qtwo - \qone$, which satisfies
    \begin{equation*}
        \begin{aligned}
            (P+\qtwo) u &= q\uone, & \  & \mbox{on } \M,\\
            u(t,x) &= 0, & \ &\mbox{for  } x\in \partial \Omega,\\
            u = 0, \  \partial_t u &= 0, & \ &\mbox{for  } t=0.\\
        \end{aligned}
    \end{equation*}
    Moreover, with the assumption that
    $\Lamlin_\qone(h) = \Lamlin_\qtwo(h)$, we have
    \[
    \partial_\nu u|_\Omega = 0.
    \]
    Now let $w = e^{i\psi}(b_0 + d_0)$ be as in Corollary \ref{cr_reminder_new},
    which solves the backward problem (\ref{eq_linearback}),
    with the phase function $\psi = i\rho^2 t + i\rho(x \cdot \varpi)$.
    Here we set $\varpi = -\omega$.
    We choose
    \[
    b_0(t, x,\varphi) = \phi(t)\prod_{j = 2}^{n} \chi(\omega_j \cdot (x - y))
    \coloneqq \phi(t)\tilde{b}_0(x, \omega).
    \]
    as a smooth function supported in a small $\ep$-neighborhood of the same ray $\gamma(s)$, with the same $\phi$.
    In this case, note that $w(T,x) = \dt w(T,x) = 0$ and the reminder term satisfies
    $\|d_0\|_{H^s((0,T) \times \Omega)} \leq C \rho^{-1}$.
    We multiply $w$ with the equation above and integrate over $t,x$ to have
    \begin{align*}
        0 = \int_0^T \int_\Omega (P+\qtwo) uw - q\uone w \diff x \diff t
        = \int_0^T \int_\Omega u (\Pad+\qtwo) w - q\uone w \diff x \diff t,
    \end{align*}
    which implies
    \[
    \int_0^T \int_\Omega - q\uone w \diff x \diff t = 0.
    \]
    The following proof is just the same as that of Proposition \ref{pp_recoverq}. For completeness, we repeat it here.
    We plug in the expansion of $\uone, w$ to have
    \begin{align*}
        \mI_q = \int_0^T \int_\Omega q a_0 b_0 \diff x \diff t
        =  -\int_0^T \int_\Omega q ( r_0 b_0 + a_0 d_0)  \diff x \diff t
        \leq C \rho^{-1}.
    \end{align*}
    When $\rho \rightarrow \infty$, this implies $\mathcal{I}_q = 0$.
    Since $\phi \in C_0^\infty((0,T))$ is arbitrary, we must have
    \[
    \int_\Omega  q(t,x)
    \tilde{a}_0 \tilde{b}_0\diff x = 0, \qquad t \in [0,T].
    \]
    Note that $\tilde{a}_0$ and $\tilde{b}_0$ are both supported in a small $\ep$-neighborhood of the ray $\gamma(s) = s \omega + y$.
    Let $\ep \rightarrow 0$ and we can extract the line integral
    \[
    X q(\gamma) = \int q(t, \gamma(s)) \diff s = 0, \qquad t \in [0,T].
    \]
    Since $\omega \in S^{n-1}$ and $y_0\in \partial \Omega$ are arbitrary,
    the X-ray transform of $q$ over all rays vanishes.
    With the assumption on $\Omega$, we have $q = \qtwo - \qone = 0$.
    Indeed, the X-ray transform is injective on $L^1(\Rn)$ by the Fourier Slice Theorem, for example, see \cite[Chapter 2]{book_SU}.
    One can extend $q \in C^\infty([0,T] \times \bar{\Omega})$ to $L^1(\Rn)$ by setting it equal to zero in $\Rn \setminus \bar{\Omega}$, as $\Omega$ is bounded.
\end{proof}
In this case, compatibility condition of the Dirichlet data $h$ for the linear problem coincides with the compatibility condition for the nonlinear one.
Indeed, both of them requires $\dt^j h(0) = 0, j = 0, \ldots, m$.
In this case, we can say
\begin{align}\label{eq_Lam}
    \Lamlin_{q}(g_0, g_1, h) = 
    \partial_\vep \Lambda_{q, \beta}(\vep g_0, \vep g_1, \vep h)|_{\vep=0},
\end{align}
for small $h \in \Hm$.
Therefore, Proposition \ref{pp_recoverq} implies the first-order linearization of $L_{q,\beta}$ determines $q$.
\subsection{The second-order Linearization}\label{subsec_2nd_new}
In this subsection, let $\vep_1, \vep_2 > 0$ be small parameters and let
$u_{\vep_1,\vep_2}$ be the solution to the nonlinear problem (\ref{eq_nl}) with data
\begin{align}\label{def_2data_new}
    g_0 = g_1 = 0, \qquad    h = \vep_1 h_1  + \vep_2  h_2.
\end{align}
We consider the second-order linearization
\[
U_2 = \partial_{\vep_1} \partial_{\vep_2} u_{\vep_1, \vep_2}|_{\vep_1 = \vep_2 = 0}.
\]
Now we have
\[
\partial_\nu U_2|_{\partial \Omega} = \partial_{\vep_1} \partial_{\vep_2}\Lambda_{q, \beta}(h)|_{\vep_1 = \vep_2 = 0}.
\]
Let $u_j$ be solutions to the linear problem
\begin{align}\label{LS_new}
    \begin{aligned}
        (P+q) u_j(t,x)  &= 0, & \  & \mbox{on } \M,\\
        u_j(t,x) &= h_j, & \ &\mbox{for  } x\in \partial \Omega,\\
        u_j = 0,\  \dt v_j &= 0, & \ &\mbox{for  } t=0.\\
    \end{aligned}
\end{align}
Then $U_2$ solves
\begin{align*}
    \begin{aligned}
        (P+q) U_2  &= \beta(t,x) \dt^2 (u_1u_2), & \  & \mbox{on } \M,\\
        U_2(t,x) &= 0, & \ &\mbox{for  } x\in \partial \Omega,\\
        U_2 = \partial_t U_2 &= 0, & \ &\mbox{for  } t=0.\\
    \end{aligned}
\end{align*}
Now for any $u_0$ solving (\ref{eq_linearback}) with the Dirichlet data $h_0$,
we integrate by parts to have
\begin{align}\label{eq_Ibeta_new}
    &\int_0^T \int_\Omega \beta(t,x) \partial_t^2(u_1 u_2) u_0 \diff x \diff t\\
    = & \int_0^T \int_\Omega (P+q) U_2 u_0 \diff x \diff t
    =  - \int_0^T \partial_\nu U_2|_{\partial \Omega} h_0
    + \dt\partial_\nu  U_2 |_{\partial \Omega} h_0 \diff t \nonumber \\
    = &- \int_0^T \partial_{\vep_1} \partial_{\vep_2}\Lambda_{q, \beta}(f)|_{\vep_1 = \vep_2 = 0} h_0
    + \dt\partial_{\vep_1} \partial_{\vep_2}\Lambda_{q, \beta}(f)|_{\vep_1 = \vep_2 = 0} h_0 \diff t
    . \nonumber
\end{align}
\subsection{Recovering $\beta$.}
Now we would like to recover the nonlinear coefficient from the DN map.
Suppose we have two  coefficients $\betaone, \betatwo \in C^\infty([0,T] \times \bar{\Omega})$.
Let $\uk_{\vep_1,\vep_2}$ be the solutions to the nonlinear problem (\ref{eq_nl})  with $\betak$, the vanishing initial conditions, and the Dirichlet data given in (\ref{def_2data_new}), for $k=1,2$.
We consider the corresponding DN maps $\Lambda_\qbk$ and suppose
\[
\Lambda_\qbone (h) = \Lambda_\qbtwo (h),
\]
for any small and compatible Dirichlet data $h$.

First, according to (\ref{eq_Lam}) and Proposition \ref{pp_recoverq_new} we have $\qone = \qtwo$.
Next, we would like to construct approximate solutions $u_j$ to linear problems using Section \ref{sec_goc} and \ref{subsec_reminder_new},
which are corresponding to the Dirichlet data $h_j$, for $j = 1,2,0$.
As before, by (\ref{eq_Ibeta_new}), we have
\begin{align}\label{eq_Ibeta_zero}
    \int_0^T \int_\Omega (\betaone- \betatwo)  \ddt(u_1 u_2) u_0 \diff x \diff t = 0.
\end{align}
Then we follow exactly the same idea as in the proof of Section \ref{subsec_beta}. But for completeness, we repeat it here.

As before, let $\omega_1, \omega_2, \ldots, \omega_{n}$ be an orthonormal basis for $\mathbb{R}^n$.
Let $\varpi =  - (\omega_1 + \omega_2)/\sqrt{2} \in S^{n-1}$ and we choose $\varpi_2, \ldots \varpi_n$ such that they form an orthonormal basis with $\varpi$.
For a fixed point $x_0 \in \Omega$,
we choose $y_1, y_2, y_0\in \partial \Omega$ such that the three rays
\[
\gamma_j(s) = s \omega_j + y_j, 
\ \text{ for } j = 1,2,
\qquad \text{and }
\gamma_0(s) = s \varpi + y_0,
\]
intersect at $x_0$.
We construct approximate solutions
\begin{align*}
    u_{j}  &=  e^{\rho^2t - \rho (w_j \cdot x)} (a_{0,j} + r_{0,j}) ,\quad j = 1,2,
\end{align*}
to $(P + \qone) u_j =(P + \qtwo) u_j  = 0$,
according to Proposition \ref{pp_reminder_new}.
Here the amplitudes given by
\begin{align*}
    a_{0,1}(t,x,\omega_j) =  \phi(t) \prod_{i \neq 1} \chi(\omega_i \cdot (x - y_1)) \coloneqq \phi(t)\tilde{a}_{0,1}(x, \omega),\\
    a_{0,2}(t,x,\omega_j) =  \phi(t) \prod_{i \neq 2} \chi(\omega_i \cdot (x - y_2)) \coloneqq \phi(t)\tilde{a}_{0,2}(x, \omega),
\end{align*}
are smooth function supported in a small $\ep$-neighborhood of the ray $\gamma_1(s)$ and $\gamma_2(s)$ respectively.
Moreover, the reminder term is given by  $\|r_{0,j}\|_{H^s((0,T) \times \Omega)} \leq C\rho^{-1}$, for $j = 1,2$, where $s$ is chosen to be large enough.
Then we set
\[
h_j = u_j|_{\partial \Omega}, \quad j =1,2,
\]
which satisfy the compatibility condition (\ref{def_comp_nl}), by our construction in Proposition \ref{pp_reminder_new}.
Then with small $\vep_1, \vep_2$, the vanishing initial conditions and the Dirichlet data $h$ given by (\ref{def_2data_new}) allows a unique solution to the nonlinear problem.
Thus, let $\uk_{\vep_1,\vep 2}$ be the solutions to the nonlinear problem (\ref{eq_nl})  with $\betak$ and (\ref{def_2data_new}).
Note the second-order linearization $U_2^{(k)}$ satisfies
\[
\partial_\nu U_2^{(1)}|_{\partial \Omega} = \partial_\nu U_2^{(2)}|_{\partial \Omega},
\]
since $\Lambda_\qbone( h) = \Lambda_\qbtwo( h)$.

Then we construction approximate solutions $u_0$ to the backward problem (\ref{eq_linearback}) by
\begin{align*}
    u_{0}  &=  e^{-2\rho^2t - \sqrt{2}\rho (\varpi \cdot x)} (b_{0} +  d_{0}),
\end{align*}
according to Corollary \ref{cr_reminder_new}.
Here the amplitudes given by
\begin{align*}
    b_{0}(t,x,\varpi) =  \phi(t) \prod_{i \neq 1} \chi(\varpi_i \cdot (x - y_0)) \coloneqq \phi(t)\tilde{b}_{0}(x, \varpi),
\end{align*}
is a smooth function supported in a small $\ep$-neighborhood of the ray $\gamma_0(s)$.
In this case, the reminder term is given by  $\|d_{0}\|_{H^s((0,T) \times \Omega)} \leq C\rho^{-1}$.

Now we plug in $u_1,u_2,u_0, U_2^{(1)}, U_2^{(2)}$ in (\ref{eq_Ibeta_new}) to have
\begin{align*}
    \int_0^T \int_\Omega (\betaone- \betatwo)  \ddt(u_1 u_2) u_0 \diff x \diff t = 0.
\end{align*}
As before, we have
\begin{align*}
    \ddt(u_1 u_2) &
    = e^{2 \rho^2 t -\rho (\omega_1 + \omega_2) \cdot x} (4 \rho^4  \phi^2(t) \tilde{a}_{0,1}\tilde{a}_{0,2} + R_1),
\end{align*}
where we write
\begin{align*}
    R_1 &= 4 \rho^4 (r_{0,1}(a_{0,2} +  r_{0,2}) + (a_{0,1} + r_{0,1}) r_{0,2})\\
    &-2 \rho^2 \dt((a_{0,1} + r_{0,1})(a_{0,2}  + r_{0,2})) + \ddt ((a_{0,1} +  r_{0,1})(a_{0,2} + r_{0,2})).
\end{align*}
A straightforward computation shows that $\|R_1\|_{H^{s-2}((0,T)\times \Omega)} \leq C \rho^3$.
Then we have
\begin{align*}
    0
    & = \int_0^T \int_\Omega (\betaone- \betatwo)    (4 \rho^4  \phi^3(t) \tilde{a}_{0,1}\tilde{a}_{0,2} \tilde{b}_{0}+ R_2) \diff x \diff t
\end{align*}
where we write
\[
R_2  = R_1 (b_{0} + d_{0})  + (4 \rho^4  \phi^2(t) \tilde{a}_{0,1}\tilde{a}_{0,2} + R_1) d_{0}.
\]
Similarly, we have $\|R_2\|_{H^{s-2}((0,T)\times \Omega)}  \leq C \rho^3$.
It follows that
\begin{align*}
    \int_0^T \int_\Omega (\betaone- \betatwo)   \phi^3(t) \tilde{a}_{0,1}\tilde{a}_{0,2} \tilde{b}_{0} \diff x \diff t = 0.
\end{align*}
Since $\phi \in C_0^\infty((0,T))$ is arbitrary, this implies
\[
\int_\Omega (\betaone (t,x)- \betatwo(t,x)) \tilde{a}_{0,1}\tilde{a}_{0,2} \tilde{b}_{0}
\diff x = 0, \qquad t \in [0,T].
\]
Recall the rays $\gamma_1, \gamma_2, \gamma_0$ intersect at the fixed point $x_0 \in \Omega$.
Then we can assume $\tilde{a}_{0,1}\tilde{a}_{0,2} \tilde{b}_{0}$ is supported in a small neighborhood $N_\ep$ of $x_0$, which is contained in a ball of radius $\ep$.
By shrinking the support of $\chi$,
we have
\[
\betaone(t,x_0)- \betatwo(t,x_0) = \lim_{\ep \rightarrow 0} \frac{1}{|N_\ep|}\int_\Omega (\betaone(t,x)- \betatwo(t,x)) \tilde{a}_{0,1}\tilde{a}_{0,2} \tilde{b}_{0}
\diff x
= 0.
\]
Since $x_0 \in \Omega$ is arbitrary, we must have $\betaone = \betatwo$ in $[0,T] \times \bar{\Omega}$.

\section{Appendix}\label{sec_apdx}
\subsection{Preliminaries}
Suppose $X$ is a Banach space and we denote by $H^k([0,T]; X)$  the set of all measurable functions $u: [0,T] \rightarrow X$ with
\[
\|u\|_{H^k([0,T]; X)} = (\sum_{j = 0}^k \int_0^T \| \dt^j u(t)\|^2_X \diff t)^{\frac{1}{2}} < +\infty.
\]
We have the following Sobolev embedding result, according to \cite[Section 5.9.2]{evans2022partial}.
\begin{lm}\label{lm_emb1}
    Let $\Omega \subset \Rn$ be a bounded open set with smooth boundary.
    For integers $l \geq 0$ and $k \geq \kn$,
    suppose $u \in H^{l+1}([0,T]; H^k ({\Omega}))$.
    Then we have
    \[
    u \in C^{l}([0,T]; H^k(\Omega)) \subset
    C^{l}([0,T]; C^{k-\kn}(\bar{\Omega}))
    \]
    with the estimates
    \[
    \sum_{j = 0}^{l} \sup_{t \in [0, T]}\| \dt^{j} u(t) \|_{C^{k-\kn}(\bar{\Omega})}
    \leq C\sum_{j = 0}^{l} \sup_{t \in [0, T]}\| \dt^{j} u(t) \|_{H^k}
    \leq C'\| u\|_{H^{l+1}([0,T]; H^k (\Omega))}.
    \]
\end{lm}
\begin{proof}
    By \cite[Section 5.9, Theorem 2]{evans2022partial}, we have
    \[
    \sup_{t \in [0, T]}\| \dt^{j} u(t) \|_{k} \leq \| u\|_{H^{l+1}([0,T]; H^k (\Omega))},
    \]
    for any $j = 0, 1, \ldots, l$.
    Then by the Sobolev embedding inequality, with  $\dt^{j} u(t) \in H^k(\Omega)$,
    we have
    \[
    \| \dt^{j} u(t) \|_{C^{k-\kn}(\bar{\Omega})} \leq C \| \dt^j u(t)\|^2_{H^k},
    \]
    which proves the desired result.
\end{proof}

\subsection{The linear problem with variable coefficients}\label{sec_linear}
Let $\Omega \subset \mathbb{R}^n$ be a bounded open set with smooth boundary, for $n \geq 2$.
In this section, we consider the following linear problem
\begin{align}\label{eq_linear_f}
    \begin{aligned}
        (1 - 2\beta v) \ddt u - \Laplace u - \dt \Laplace u -q(t,x)u &= f(t,x), & \  & \mbox{on } (0,T)\times \Omega,\\
        u(t,x) &= 0, & \ &\mbox{for  } x\in \partial \Omega,\\
        u = g_0, \quad & \partial_t u = g_1 , & \ &\mbox{for  } t=0,
    \end{aligned}\end{align}
where $\beta, q \in C^\infty([0,T]\times \bar{\Omega})$, $f$ is the source,  and $v$ is a function to be specified later.

For the purpose of the nonlinear problem, we consider the space $\zm(R,T)$, the set containing all functions $u$ such that
\[
u \in \bigcap_{k=0}^{m} C^{m-k}([0,T]; H^{2k+1}(\Omega)),
\quad
\|u\|^2_{\zm} = \sum_{k=0}^m \sup_{s \in [0, T]}   \|\partial_t^{m-k} v(s)\|^2_{{H^{2k+1}}} \leq R^2,
\]
where $R> 0$ is a parameter to be chosen later.

To deal with the initial value problem with a source, we consider
\begin{align}\label{def_gf}
g_0 \in H_0^1(\Omega) \cap H^{2m+1}(\Omega), \qquad
g_1 \in H_0^1(\Omega) \cap H^{2m-1}(\Omega),\qquad
f \in \zmm(R, T),
\end{align}
for some $R > 0$.
For convenience, we write
$\Cjl = \binom{l-1}{j}$.
For $l = 1, \ldots, m-1$, suppose $1- 2\beta(0)v(0) \neq 0$.
We recursively define
\begin{align}\label{def_gl}
    g_{l+1} \coloneqq  (1- 2\beta(0)v(0))^{-1}(\dt^{l-1} f(0) + G_l(0, x, v(0), \dt v(0), \ldots, \dt^{l-1}v(0)); g_0, \ldots, g_l),
\end{align}
where we introduce the notation
\begin{align}\label{def_Gl}
&G_l(t,x, v, \dt v, \ldots, \dt^{l-1}v; g_0, \ldots, g_l)\\
=&-\sum_{j=2}^{l}\Cjl\dt^{l+1-j} (1-2\beta(t) v(t)) g_j
+ \Laplace g_{l-1}
+  \Laplace g_l
- \sum_{j=0}^{l-1}\Cjl\dt^{l-1-j}q(t) g_j,\nonumber
\end{align}
We say $g_0, g_1, f$ satisfy the \textit{$m$th-order compatibility condition}, if
\begin{align}\label{def_comp_linear}
g_{l+1} \in H_0^1(\Omega) \cap H^{2(m-l) - 1}(\Omega),
\qquad \text{for }l = 1, \ldots, m-1.
\end{align}
when $a(0) \neq 0$ and $\dt^{l-1} a(0) \in C(\bar{\Omega})$ is valid,
for $l = 1, \ldots, m$.
For convenience, we write
\begin{align}\label{def_Ngf}
\Ngf = \|f\|_\zmm + \sum_{l=0}^m \|g_l\|_{H^{2(m-l)+1}}.
\end{align}

\begin{pp}\label{pp_energy}
    Let $m \geq 2k_n$ and $T > 0$ be fixed. Let $q, \beta \in C^\qCm([0,T]\times \bar{\Omega})$.
    Suppose $g_0, g_1$ and $f$ satisfy (\ref{def_gf}) and the $m$th-order compatibility condition defined as above.
    Then there exists $\epv> 0$ {depending on $\beta$ and the domain $\Omega$},
    such that for any $v \in \zm(\epv, T)$,
    the linear problem (\ref{eq_linear_f}) has a unique solution
\begin{align*}
             u \in \bigcap_{k=0}^{m} C^{m-k}([0,T]; H^{2k+1}(\Omega)),
\end{align*}
satisfying
\begin{align}\label{energy}
\|u\|^2_{\zm} \leq C (\|f\|_\zmm + \|g_0\|_{H^{2m+1}} + \|g_1\|_{H^{2m-1}}),
\end{align}
where $C$  depends on $m$, $T$, $\beta$, $q$, and the domain $\Omega$.
\end{pp}
\begin{proof}
    First, we would like to prove there exists a solution $u$ such that
    \begin{align}\label{eq_Hm}
        u \in \bigcap_{k=0}^{m+1} H^{m+1-k}([0,T]; H^{2k+1}(\Omega)),
        \quad   \text{with } \sum_{k=0}^{m+1} \int_0^T  \|\partial_t^{m+1-k} v(s)\|^2_{{H^{2k+1}}} \diff s \leq  C \Ngf.
    \end{align}
    Then by Lemma \ref{lm_emb1} and an estimate for $\Ngf$, we have the desired result.

    In the following, recall we write $a(t) = 1 - 2\beta v$ and  $C$ denotes a generic positive constant that depends on $m$, $T$, $\beta$, $q$, $\epv$, and the domain $\Omega$.
    With $\beta \in C^\infty([0,T]\times \bar{\Omega})$, we choose $\epv$ small enough such that $a$ satisfies
    \[
    \frac{1}{2}\leq  a(t, x) \leq \frac{3}{2}, \quad \text{for any } (t,x) \in [0, T] \times \Omega.
    \]
    Moreover, by Lemma \ref{lm_emb1}, we have
        \begin{align}\label{eq_at}
            \dt^{l} a \in C([0,T]; C^{k}(\bar{\Omega})), \quad \text{ if } 2l+k \leq 2m-\kn + 1.
            \end{align}

    We use the Galerkin approximation method and construct a sequence of approximate solutions $\un(t)$ to the equation above.
    Let $0= \lambda_0< \lambda_1\leq \lambda_2<\cdots$ be the eigenvalues $-\Laplace$ with vanishing Dirichlet boundary condition, listed according to their multiplicity, with the corresponding orthonormal eigenfunctions $\phi_k$ in $L^2(\Omega)$, where $k\in \mathbb{N}$.
    We consider  approximate solutions $\un(t)$ given by $\un(t)= \sum_{k=1}^i u_{i,k}(t) \phi_k$, which satisfies
    \begin{align}\label{eq_dtl}
    &\lge  a(t) \dt^{l+1} \un , w \rge +
    \lge \sum_{j=2}^{l} \Cjl\dt^{l+1-j} a(t) \dt^{j} \un , w \rge - \lge \dt^{l-1}\Laplace \un , w \rge
    - \lge \dt^l \Laplace \un , w \rge \\
    &\qquad \qquad \qquad \qquad \qquad \qquad\qquad \qquad +\lge \sum_{j=0}^{l-1}\Cjl\dt^{l-1-j}q \dt^j \un, w \rge
    =  \lge \dt^{l-1}f(t,x), w \rge, \nonumber
    \end{align}
    for any $t \in [0, T]$ and any $w$ in the space spanned by $\phi_1, \ldots, \phi_i$.
    To get this, we formally differentiate the original equation $l-1$ times with respect to $t$ and the initial conditions are
    \[
    \dt^j \un(0) = g_j, \quad j= 0, 1, \ldots, l,
    \]
    where $g_j$ is defined in (\ref{def_gl}).
    Note when $l = 1$, we do not have the second term.
    There exists a solution $u_{i,k}(t)$ to the ODE obtained from the equation above.
    We derive prior energy estimates for $\un(t)$ in the following.

    \noindent\textbf{Step 1.} We set $w = \dt^l \un$ and we integrate it  with respect to $t$.
    We estimate each term below.
    From the first term, we have
    \begin{align*}
        \int_0^t \lge a(s) \dt^{l+1} \un(s), \dt^l \un(s) \rge \diff s
        &=
        \frac{1}{2} \lge a(s) \dt^l \un(s), \dt^l \un(s) \rge|_0^t -\frac{1}{2} \int_0^t \lge\dt a(s) \dt^l \un(s), \dt^l \un(s) \rge \diff s,\\
        & \geq \frac{1}{4} \|\dt^l \un(t)\|_\ltwo^2
        - \frac{3}{4} \|\dt^l \un(0)\|_\ltwo^2
        - C \int_0^t \|\dt^l \un(s)\|_\ltwo^2 \diff s.
    \end{align*}
    Recall we have the {compatible initial conditions $\dt^l \un(0) = g_l \in H^1_0(\Omega)$.}
    Next, we estimate
    \begin{align*}
        -\int_0^t \lge \dt^{l-1} \Laplace \un(s) , \dt^l \un(s) \rge \diff s
        &= \frac{1}{2}\|\dt^{l-1} \nabla \un(t)\|_\ltwo^2 - \frac{1}{2}\|\dt^{l-1} \nabla \un(0)\|_\ltwo^2,\\
        -\int_0^t \lge \dt^l \Laplace \un(s) , \dt^l \un(s) \rge \diff s
        &=
        \int_0^t  \|\dt^{l} \nabla \un(s) \|_\ltwo^2 \diff s,\\
         -\int_0^t \lge \sum_{j=0}^{l-1}\Cjl\dt^{l-1-j}q \dt^j \un, \dt^l \un(s) \rge
        & \leq
        C\sum_{j=0}^{l} \int_0^t \|\dt^j u(s)\|_\ltwo \diff s,
    \end{align*}
   and
    \begin{align*}
        \int_0^t \lge \dt^{l-1} f(s) , \dt^l \un(s) \rge \diff s
        \leq \int_0^t  \|\dt^{l-1} f(s)\|_\ltwo^2 + \|\dt^{l} \un(s)\|_\ltwo^2 \diff s.
    \end{align*}
To get an estimate for the second term of (\ref{eq_dtl}), we consider two different cases,  when $l \leq L$ and $l \geq  L+1$,
where we define
\begin{align}\label{eq_L}
    L = m-\Ckn + 1,
\end{align}
with $\kn$ defined in (\ref{eq_kn}).
In the first case, we have
$l+1-j \leq m-\Ckn$ for any $j \geq 2$, which implies
$\dt^{l+1-j}a \in C([0,T]; C(\bar{\Omega}))$ by (\ref{eq_at}).
Then we have
\begin{align}\label{eq_modify1}
    &\sum_{j=2}^{l}  \int_0^t \lge \Cjl \dt^{l+1-j} a(s) \dt^{j} \un(s) , \dt^l \un(s) \rge \diff s
     \leq C\sum_{j=2}^{l}  \int_0^t \| \dt^{j} \un(s) \|_\ltwo^2 \diff s,
\end{align}
for $l = 1, \ldots, L$.
    It follows that
    \begin{align*}
        \frac{1}{4}  \|\dt^l \un(t)\|_\ltwo^2 + \frac{1}{2}\|\dt^{l-1} \nabla \un \|_\ltwo^2  + \int_0^t  \|\dt^{l} \nabla \un(s) \|_\ltwo^2 \diff s
        \leq
         \Ngf
        + C  \sum_{j=2}^{l}  \int_0^t \| \dt^{j} \un \|_\ltwo^2 \diff s.
    \end{align*}
    We summarize over $l=1, \ldots, L$ and use the Gr\"onwall inequality to have
    \begin{align*}
        \sum_{l=1}^{L}   \|\dt^l u(t)\|_\ltwo^2 + \|\dt^{l-1} \nabla \un \|_\ltwo^2  +  \int_0^t  \|\dt^{l} \nabla \un(s) \|_\ltwo^2 \diff s
        \leq
        C\Ngf. 
    \end{align*}
    In particular, using the Poincar\'e inequality, we have
    \begin{align}\label{eq_dtlu}
        \sum_{l=0}^{L}   \int_0^t   \|\dt^l u(s)\|_\ltwo^2  +  \|\dt^{l} \nabla \un(s) \|_\ltwo^2 \diff s
        \leq
        C \Ngf.
    \end{align}
    When $l \geq L+1$, we need a new estimate for (\ref{eq_modify1}) using higher order regularity.

    \noindent\textbf{Step 2.} We would like to derive higher-order regularity estimates.
    More explicitly, for $l = 1, \ldots, L$, we rewrite (\ref{eq_dtl}) as
    \begin{align}\label{eq_dtl_new}
        \lge \dt^{l-1}(-\Laplace) \un, w \rge
        +  \lge \dt^l(-\Laplace) \un , w \rge
        = -\lge & a(t) \dt^{l+1} \un, w \rge
        - \lge \sum_{j=2}^{l} \Cjl  \dt^{l+1-j} a(t) \dt^{j} \un, w \rge\\
        &- \lge \sum_{j=0}^{l-1}\Cjl\dt^{l-1-j}q \dt^j \un, w \rge
        +  \lge  \dt^{l-1}f(t,x), w \rge,\nonumber
    \end{align}
    where we set $w = \dt^{l} (-\Laplace)^{2k} \un$,
    for non-negative integer $k$ satisfying ${k+l \leq L}$.
    By (\ref{eq_at}), it follows that
    \begin{align}\label{eq_a_ljk}
    {\dt^{l-1}a \in C([0, T]; C^{2k}(\bar{\Omega})).}
    \end{align}
    We have
    \begin{align*}
        \lge \dt^{l-1}(-\Laplace) \un, \dt^{l} (-\Laplace)^{2k} \un \rge
        &= \frac{1}{2} (\dt \| \dt^{l-1} (-\Laplace)^{k} \nabla \un\|_\ltwo^2), 
        \\
        \lge \dt^l (-\Laplace)  u , \dt^{l} (-\Laplace)^{2k} \un \rge
        &=
        \| \dt^{l} (-\Laplace)^{k} \nabla \un \|_\ltwo^2, 
    \end{align*}
    where we integrate by parts and using the property that $((-\Laplace)^j \un) |_{\partial \Omega} = \sum_{k=1}^i \lambda_k^j \phi_k|_{\partial \Omega}  = 0$.
    {Recall the initial condition $\dt^{l-1} u(0) = g_{l-1} \in H^{2(m-l)+3}(\Omega)$. }
    Moreover, we have
    \begin{align}\label{eq_modify3}
        &-\lge a \dt^{l+1} \un, \dt^{l} (-\Laplace)^{2k} \un \rge \\
        =&  \lge \nabla (-\Laplace)^{k-1}(a \dt^{l+1} \un), \dt^{l} \nabla(-\Laplace)^{k} \un\rge \nonumber \\
        \leq  &C \|\dt^{{l+1}} \un\|_{H^{2k-1}}^2 + \frac{1}{4}\|\dt^{{l}} \un\|_{H^{2k+1}}^2. \nonumber
    \end{align}
    In addition, using (\ref{eq_a_ljk}), we have
    \begin{align}\label{eq_modify4}
        & - \sum_{j=2}^{l}   \lge \Cjl \dt^{l+1-j} a \dt^{j} \un, \dt^{l} (-\Laplace)^{2k}  \un \rge \\
        = & -  \lge \sum_{j=2}^{l}  \Cjl \nabla (-\Laplace)^{k-1} (\dt^{l+1-j} a(t) \dt^{j} \un) , \dt^{l} \nabla (-\Laplace)^k \un \rge \nonumber \\
        \leq &
        C\sum_{j=2}^{l} \| \dt^{j} \un\|_{H^{2k-1}}^2 + \frac{1}{4}\|\dt^{l} \un\|_{H^{2k+1}}^2\nonumber,
    \end{align}
    and
    \begin{align*}
        \lge \dt^{l-1} f, \dt^{l} (-\Laplace)^{2k} \un \rge
        \leq & 4\| \dt^{l-1} (-\Laplace)^{k-1}\nabla  f\|_\ltwo^2 + \frac{1}{4}\| \dt^{l} (-\Laplace)^{k} \nabla  \un\|_\ltwo^2\\
        \leq & 4\| \dt^{l-1}  f\|_{H^{2k-1}}^2 +\frac{1}{4}\| \dt^{l}   \un \|_{H^{2k+1}}^2.
    \end{align*}
    Combining the inequalities above, we have
    \begin{align}\label{eq_induction_even}
        \frac{1}{4}\int_0^t \| \dt^{l} \un(s)\|_{H^{2k+1}}^2 \diff s
        \leq
        C (\sum_{j=2}^{l+1} \int_0^t     \| \dt^{j} \un(s) \|_{H^{2k-1}}^2
        \diff s + \Ngf),
    \end{align}
    for any non-negative integer $l + k \leq L$.
%

     Further, we would like to use an inductive procedure to show
    \begin{align}\label{rm_Hmk}
    \sum_{l=0}^{L-k}\int_0^t \| \dt^{l} \un (s)\|_{H^{2k+1}}^2 \diff s
    \leq C \Ngf,
    \quad \text{for any } k = 0, \ldots,  L.
    \end{align}
    Indeed, this is true for $k=0$, since in (\ref{eq_dtlu}) we have the estimates for $\|\dt^{l} \un \|_\hone$, when $l = 0, \ldots, L$.
    Let $k' \geq 1$ and assume (\ref{rm_Hmk}) holds for $k \leq k'$.
    We would like to prove it is also true for $k = k'+1$.
    Using (\ref{eq_induction_even}), we have
     \begin{align*}
        \int_0^t \| \dt^{l} \un(s)\|_{H^{2k'+3}}^2 \diff s
        \leq C (\sum_{j=2}^{l+1} \int_0^t \| \dt^{j} \un(s) \|_{H^{2k' + 1}}^2
        \diff s + \Ngf)
    \end{align*}
     if $l \leq L - k'-1$.
    By the assumption, this is bounded by $C\Ngf$, since $j \leq L-k'$ when $j \leq l+1$.
    Thus, we have the desired estimate for $k'+1$. By induction, we prove (\ref{rm_Hmk}).

   \noindent\textbf{Step 3.} We would like to finish Step 1 and prove
    (\ref{eq_dtlu}) for $L +1 \leq l \leq m+1$.
    In this case, we rewrite (\ref{eq_modify1}) as
    \begin{align*}
    &\sum_{j=2}^{l}  \int_0^t \lge \Cjl \dt^{l+1-j} a(s) \dt^{j} \un(s) , \dt^l \un(s) \rge \diff s\\
    =&\int_0^t \sum_{j=l-L+2}^{l}   \lge \Cjl \dt^{l+1-j} a(s) \dt^{j} \un(s) , \dt^l \un(s) \rge
    + \sum_{j=2}^{l-L+1} \lge \Cjl \dt^{l+1-j} a(s) \dt^{j} \un(s) , \dt^l \un(s) \rge \diff s.
    \end{align*}
When $j \geq l-L+2$, by (\ref{eq_at}) and (\ref{eq_L}) we have $\dt^{l+1-j} a \in C([0,T]; C(\bar{\Omega}))$.
In this case, one can estimate the first term as before.
When $j \leq l-L+1$,
we have
\[
j \leq (m+1)-L+1  = \lceil (\kn+1)/2 \rceil \leq L - 1,
\]
since $m \geq 2 \kn$.
By Lemma \ref{lm_emb1}, this implies
    \begin{align*}
    \int_0^t \|\dt^{l+1-j} a(s) \dt^{j} \un(s) \|_\ltwo  \diff s
    &\leq \|\dt^j \un\|_{ C([0,T]; C(\Omega))}   \|\dt^{l+1-j} a\|_\ltwo\\
    &\leq C\|\dt^j \un\|_{H^1([0,T]; H^{2k+1}(\Omega))},
\end{align*}
where we set $k = \lceil (\kn-1)/2 \rceil$.
With $j \leq l-L+1$,
note that
\[
j + k \leq \kn + 1 \leq L,\]
where the last inequality holds as $m \geq 2\kn$.
By (\ref{rm_Hmk}),
we have $\|\dt^j \un\|_{H^1([0,T]; H^{2k+1}(\Omega))} \leq C\Ngf$,
which implies
\begin{align*}
    &\sum_{j=2}^{l}  \int_0^t \lge \Cjl \dt^{l+1-j} a(s) \dt^{j} \un(s) , \dt^l \un(s) \rge \diff s
    \leq C(\sum_{j=l-L+2}^{l}  \int_0^t \| \dt^{j} \un(s) \|_\ltwo \diff s
    + \Ngf).
\end{align*}
The rest of terms have the same estimates as before.
This proves a complete version of (\ref{eq_dtlu}), i.e.,
\begin{align}\label{eq_dtlu_new}
    \sum_{l=0}^{m+1}  \int_0^t  \|\dt^{l} \un \|_\ltwo^2 \diff s
    + \int_0^t  \|\dt^{l} \nabla \un \|_\ltwo^2 \diff s
    \leq
    C \Ngf.
\end{align}

\noindent\textbf{Step 4.} Thus, from (\ref{eq_dtlu_new}), we can conclude that $\{\un\}_{n=1}^\infty$ is bounded in $H^{m}([0,T]; H^1_0(\Omega))$, with desired estimates.
We may extract a subsequence which converges weakly to some
\[
u \in  H^{m+1}([0,T]; H^1_0(\Omega)).
\]
By a standard argument, we can show $u$ is a weak solution to (\ref{eq_linear}).

    Last, we would like to show  such $u$ is in $H^{m+1-k}([0,T]; H^{2k+1}(\Omega))$ with the desired estimate (\ref{eq_Hm}) for $k = 0, \ldots, m$ by an inductive procedure.
    Following the same proof of (\ref{eq_dtlu_new}) and passing to limits as $i \rightarrow +\infty$, this statement is true for $k = 0$.
    Now we prove by induction.
    Assume that $u \in H^{m+1-k}([0,T]; H^{2k+1}({\Omega}))$ satisfies
    \begin{align}\label{assump_Hmk}
        \sum_{j=0}^{m+1-k}  \int_{0}^t  \|\partial_t^{l} u(s)\|^2_{H^{2k+1}} \diff s
        \leq C\Ngf,
    \end{align}
    for any $k \leq k' -1 $,
    where $1 \leq k' \leq m+1$ is an integer.
    In the following, we would like to prove $u \in H^{m-k'}([0,T]; H^{2k'+1}(\Omega))$
    with the desired estimate.
    It suffices to show that
    \[
    \int_{0}^t  \|\partial_t^{l} u(s)\|^2_{H^{2k'+1}} \diff s
    \leq C\Ngf,
    \]
    for any $l +k' \leq m+1$.
    We prove this for $\un(t)$ and then pass to limits as $i \rightarrow +\infty$ as before.
    Indeed, we consider (\ref{eq_dtl_new}) and set $w = \dt^{l} (-\Laplace)^{2k'} \un$.
    Previously, we have shown this conclusion for $k'+ l \leq L$, see (\ref{eq_induction_even}).

    Now we consider the case when $l + k' \leq m+1$ but $l + k' \geq L+1$.
    We follow the same estimates as before, except for (\ref{eq_modify3}) and (\ref{eq_modify4}).
    Recall for (\ref{eq_modify3}) and  (\ref{eq_modify4}), actually we need to derive an estimate for the following terms
    \begin{align*}
        &\sum_{j = 2}^{l+1} \int_0^t  \|\nabla (-\Laplace)^{k'-1}(\dt^{l+1-j} a(s) \dt^{j} u(s)) \|_\ltwo^2 \diff s.
    \end{align*}
    Let $k_1, k_2$ be non-negative integers satisfying $k_1 + k_2 = 2k'-1$.
    According to (\ref{eq_at}), we have $\dt^{l+1-j} a \in C([0,T];C^{k_1}(\bar{\Omega}))$ if
    \begin{align*}
    2(l+1-j) + k_1 \leq 2m-\kn+1 \quad \Rightarrow \quad 2l-2j + k_1 \leq 2m - \kn-1.
    \end{align*}
    On the other hand,  we have $\dt^{j} u \in C([0,T];C^{k_2}(\bar{\Omega}))$ if
    \begin{align*}
        2j + k_2 \leq 2m-\kn +1.
    \end{align*}
    Note that
    \[
    (2l-2j + k_1) + (2j + k_2) = 2l+2k'-1 \leq 2m + 1 \leq 4m -2\kn
    \]
    is always true, as $m \geq 2 \kn$.
    This implies we have either  $\dt^{l+1-j} a \in C([0,T];C^{k_1}(\bar{\Omega}))$ or $\dt^{j} u \in C([0,T];C^{k_2}(\bar{\Omega}))$.
    In the first case, i.e., $2l-2j + k_1 \leq 2m - \kn-1$,
    we derive the estimate as before.
    Otherwise, suppose $2l-2j + k_1 \geq 2m - \kn$.
    On the one hand, we need $\dt^j u \in H^1([0,t]; H^{k_2 + \kn}(\Omega))$.
    On the other hand,
    note that we have
    \[
    2j + k_2 \leq (2l+2k'-1) - (2m - \kn) \leq  \kn + 1,
    \]
    which implies
    \begin{align*}
    (j+1) + \lceil{(k_2+\kn-1)}/{2}\rceil
    &\leq (j+1) + (k_2+\kn-1)/{2} + {1}/{2}\\
    & = (2j + k_2)/2 + \kn/2 + 1
    < \kn + 2.
    \end{align*}
    To use (\ref{rm_Hmk}), recall $m \geq 2 \kn$ such that
    $\kn + 1 \leq L$.
    Hence, we use Lemma \ref{lm_emb1} and (\ref{rm_Hmk}) to have
    \[
    \|\dt^j u \|_{C([0,T];C^{k_2}(\bar{\Omega}))} \leq C\|\dt^j u \|_{H^1([0,T];H^{k_2 + \kn}(\Omega))} \leq C\Ngf.
    \]
    To conclude, we have
    \begin{align*}
        \int_0^t \| \dt^{l} u(s)\|_{H^{2k'+1}}^2 \diff s
        \leq
        C (\sum_{j=2}^{l+1} \int_0^t     \| \dt^{j} \un(s) \|_{H^{2k'-1}}^2
        \diff s + \Ngf),
    \end{align*}
    for any non-negative integers $l+k' \leq m+1$.
    By induction, we have (\ref{eq_Hm}).

    Note with $v \in \zm(\epv,T)$ and $f \in \zmm(R,T)$, a similar inductive procure shows from (\ref{def_gl}) we have
    \[
    \|g_{l+1}\|_{H^{2(m-l-1)+1}} \leq C(\|f\|_\zmm + \|g_0\|_{H^{2m+1}} + \|g_1\|_{H^{2m-1}}), \quad l  = 1, \ldots, m-1,
    \]
    for a constant $C$ depends on $m$, $\beta$, $q$, and the domain $\Omega$.
    Then by Lemma \ref{lm_emb1}, we prove the result.
\end{proof}
For the nonlinear problem, we need the following lemma.
\begin{lm}\label{lm_compt}
    For $v \in \tzm(r,T)$, the initial conditions $\tg_0, \tg_1$ and the source $F(u_h, v)$ satisfy the $m$th-order compatibility condition (\ref{def_comp_linear}) for the linearized equation (\ref{eq_linear_w}).
\end{lm}
\begin{proof}
    Indeed, we need to check whether
    \begin{align*}
        \hat{g}_{l+1} \coloneqq
        &(1-2\beta(0)(v(0)+u_h(0)))^{-1}(\dt^{l-1} F(0) \\
        &+ G_l(0, x,v(0)+u_h(0), \dt (v(0)+u_h(0)), \ldots, \dt^{l-1}(v(0)+u_h(0))); \hat{g}_0, \ldots, \hat{g}_l)
    \end{align*}
    belongs to $H_0^1(\Omega) \cap H^{2(m-l)-1}(\Omega)$ or not,
    where we have $\hat{g}_0 = \tg_0$, $\hat{g}_1 = \tg_1$, and $G_l$ defined in (\ref{def_Gl}).

    Note with the initial conditions $\dt^l v(0) = \tg_l$, one has $\dt^l (v(0) + u_h(0)) =g_l$,
    for $l = 0, \ldots, m$.
    This implies
    \begin{align*}
        \hat{g}_{l+1} =
        (1-2\beta(0)g_0)^{-1}(\dt^{l-1} F(0)
        + G_l(0, x, g_0, g_1, \ldots, g_{l-1}; \hat{g}_0, \ldots, \hat{g}_l)
    \end{align*}
    In the following, we would like  to show $\hat{g}_{l+1}= \tg_{l+1}$ also holds for $l  = 1, \ldots, m-1$, by an inductive procedure.
    Assume this is true for any $l \leq L-1$, where $L \geq 1$.
    We check the case $l = L$.
    We compute
    \begin{align*}
        \dt^{l-1}F(u_h, v)
        = &-(1-2\beta(v+u_h))\dt^{l+1} u_h
        + 2 \dt^{l-1}(\beta (\dt(v+u_h))^2)\\
        &+ G_l(t,x, v+u_h, \dt (v+u_h), \ldots, \dt^{l-1}(v+u_h); h(0), \ldots, \dt^l h(0)).
    \end{align*}
    Note that $G_l$ is linear in the second group of arguments.
    It follows that
    \[
    \hat{g}_{l+1} =- \dt^{l+1}h(0) + (1-2\beta(0)g_0)^{-1}( \tilde{G}(0,x,g_0, \ldots, g_l)),
    \]
    where we compare (\ref{def_tG}) and (\ref{def_Gl}).
    Then by definition we have  $\hat{g}_{l+1} = \tg_{l+1}$ is true for $l = L$.
    By induction, this is true for $l = 1, \ldots, m-1$.
    As $\tg_l \in H_0^1(\Omega) \cap H^{2(m-l)+1}(\Omega)$, we have the desired result.
\end{proof}

\bibliography{local_ma}
\bibliographystyle{plain}

\end{document}